\documentclass[12pt]{amsart}
\usepackage{amssymb,amsmath}
\usepackage{geometry}    
\usepackage{mathrsfs}

\usepackage{vmargin}
\setmargrb{1in}{1in}{1in}{1in}

\newtheorem{theorem}{Theorem}[section]
\newtheorem{lemma}[theorem]{Lemma}
\newtheorem{proposition}[theorem]{Proposition}
\newtheorem{corollary}[theorem]{Corollary}
\theoremstyle{definition}
\newtheorem{definition}[theorem]{Definition}
\newtheorem{remark}[theorem]{Remark}

\begin{document}

\title{Generalizations of a result of Jarn\'ik on simultaneous approximation}

\author{Johannes Schleischitz} 

\address{Institute of Mathematics, Univ. Nat. Res. Life Sci. Vienna, Austria}

\begin{abstract}
Consider a non-increasing function $\Psi$ from the positive reals to the positive reals
with decay $o(1/x)$ as $x$ tends to infinity. Jarn\'ik proved in 1930 that there exist real numbers
$\zeta_{1},\ldots,\zeta_{k}$ together with $1$ linearly independent over $\mathbb{Q}$
with the property that all $q\zeta_{j}$ have distance to the nearest integer smaller than
$\Psi(q)$ for infinitely many positive integers $q$, but not much smaller in a very strict sense.  
We give an effective generalization of this result to the case of successive powers of real $\zeta$.
The method also allows for generalizing corresponding results for $\zeta$ contained
in special fractal sets such as the Cantor set.
\end{abstract}

\maketitle

{\footnotesize{Supported by the Austrian Science Fund FWF grant P24828.} \\

Math Subject Classification 2010: 11J13, 11J25, 11J82}   

\vspace{2mm}

\section{Best constants for approximation of real numbers}

\subsection{Introduction}
We study the simultaneous approximation properties of vectors of the
form $(\zeta,\zeta^{2},\ldots,\zeta^{k})$ for real $\zeta$, by rational numbers
with coinciding denominator.
Dirichlet's Theorem on simultaneous Diophantine approximation asserts that for any given
$\underline{\zeta}=~(\zeta_{1},\ldots,\zeta_{k})\in{\mathbb{R}^{k}}$, the inequality
\begin{equation} \label{eq:diri}
\max_{1\leq j\leq k} \vert q\zeta_{j}-p_{j}\vert \leq q^{-\frac{1}{k}}
\end{equation}
has a solution $(q,p_{1},\ldots,p_{k})\in{\mathbb{N}\times\mathbb{Z}^{k}}$ with
arbitrarily large $q$, where $\mathbb{N}=\{1,2,\ldots\}$ throughout. Clearly, assuming $q>0$ is no restriction, 
since we may multiply any vector $(q,p_{1},\ldots,p_{k})$ by $-1$ without affecting the absolute values
in \eqref{eq:diri}. Furthermore, if $q$ tends to infinity so do the $p_{j}$ and vice versa. We
will not explicitly mention these facts in the sequel in similar settings.
Property \eqref{eq:diri} is in particular true for vectors of successive powers, i.e. $\zeta_{j}=\zeta^{j}$,
which are of particular interest in this paper. 

A question studied by Jarn\'ik, in the general setting of vectors
$\underline{\zeta}\in{\mathbb{R}^{k}}$, can be roughly explained as follows.
Consider a function $\Psi:\mathbb{R}_{>0}\mapsto \mathbb{R}_{>0}$ 
that decreases sufficiently fast that \eqref{eq:diri} can be satisfied.
Is it possible to find vectors $\underline{\zeta}\in{\mathbb{R}^{k}}$ for which 
\begin{equation} \label{eq:ni}
\max_{1\leq j\leq k} \vert q\zeta_{j}-p_{j}\vert \leq \Psi(q)
\end{equation}
has arbitrarily large solutions $(q,p_{1},\ldots,p_{k})\in{\mathbb{N}\times\mathbb{Z}^{k}}$, 
but this is no longer true if we replace $\Psi$ by some certain slightly smaller function $\psi$, 
i.e. for which $\psi(x)<\Psi(x)$ for all large $x$. He made the usual
additional assumption that $\underline{\zeta}$ is $\mathbb{Q}$-linearly independent
together with $1$. In particular the case $\psi(x)=c\Psi(x)$
for some constant $c\in{(0,1)}$ is of interest. For precise definitions see Section~\ref{formulation}.
Jarn\'ik established results on this question in~\cite{jarnik}. 
Satz~5 in~\cite{jarnik} establishes results in the somehow most general case
$\Psi(x)=o(x^{-1/k})$, Satz~6 provides stronger results considering only functions $\Psi(x)=o(x^{-1})$,
where $x\to\infty$ is meant in both cases. We will formulate Satz~6 and an immediate corollary to Satz~5
in Section~\ref{formulation}.
The present paper aims to generalize Satz~6 to the case of successive powers
in several ways. Moreover, Section~\ref{vier} deals with simultaneous approximation of
numbers in fractal sets, and complements a result by Bugeaud~\cite{buge} in various ways.

\subsection{Notation and known results}  \label{formulation}
We use a similar notation to the one in~\cite{buge} in the sequel. However, we prefer
to use the notion of linear forms instead of rational approximations, which essentially
implies a change in the $x$-exponent of the function $\Psi(x)$ we will consider to the one in~\cite{buge} by $1$. 

Throughout, let $\Psi:\mathbb{N}\mapsto \mathbb{R}_{>0}$,
i.e. $(\Psi(q))_{q\geq 1}$ induces a sequence of positive reals numbers. 
In the sequel one may always consider a continuation to a function
$\Psi:\mathbb{R}_{>0}\mapsto \mathbb{R}_{>0}$ with corresponding properties, 
however only the integer evaluations will be of interest
and for technical reasons we restrict to $\mathbb{N}$. 
For our purposes, we will consider functions $\Psi$ which satisfy some condition concerning 
both decay and local monotonicity. More precisely, we will consider $\Psi$ to have
some (A)-property and some (B)-property, the most frequently used are defined as follows.
\begin{align}
\Psi(x)&= o(x^{-1}), \qquad x\to\infty    \label{eq:bed1},  \tag{A1} \\
\Psi(x)&< dx^{-1}, \qquad 
 \text{for some fixed small} \quad d>0, \quad \text{and all} 
\quad x\geq \hat{x}, \label{eq:bed5}      \tag{A2} \\
 \Psi(x)&< \frac{1}{2}x^{-1}, \qquad  \text{for all} \quad x\geq \hat{x}, \label{eq:bed4}  \tag{A3}    
\end{align}  
and
\begin{align}
x&\leq y \quad \Longrightarrow \quad \Psi(x)\geq \Psi(y),   \label{eq:bed2}  \tag{B1}   \\
\Psi(lx)&\leq l\Psi(x), \qquad l,x\in{\mathbb{N}}.  \label{eq:bed3} \tag{B2}
\end{align}  
Condition \eqref{eq:bed5} depends on $d$ and is apriori not an exact definition.
An effective constant $d$ will appear in the context of the results, however. Assuming $d<1/2$,
it is evident that $\eqref{eq:bed1}\Longrightarrow \eqref{eq:bed5} \Longrightarrow \eqref{eq:bed4}$,
and obviously $\eqref{eq:bed2}\Longrightarrow \eqref{eq:bed3}$. Define
\[
\mathscr{K}_{1}(\Psi)= \{\zeta\in{\mathbb{R}\setminus{\mathbb{Q}}}: \vert \zeta q-p\vert\leq \Psi(q)
\text{ for infinitely many} \quad (q,p)\in{\mathbb{N}\times\mathbb{Z}}\},
\]
and let $\mathscr{K}_{1}^{\ast}(\Psi)$ be such as $\mathscr{K}_{1}(\Psi)$ but with the restriction
of relatively prime vectors $(q,p)$. Observe that assuming $\Psi$ tends to $0$ and satisfies
\eqref{eq:bed3}, we have $\mathscr{K}_{1}(\Psi)=\mathscr{K}_{1}^{\ast}(\Psi)$.
Indeed, if $(q,p)$ satisfies the corresponding inequality so does $(q^{\prime},p^{\prime})=(q/d,p/d)$, 
where we have put $d=\rm{gcd}(p,q)$, since
\begin{equation}  \label{eq:vervielfachen}
\vert \zeta q-p\vert=\vert \zeta dq^{\prime}-dp^{\prime} \vert= d\left\vert q^{\prime}\zeta-p^{\prime}\right\vert.
\end{equation}
As $\zeta\notin{\mathbb{Q}}$, infinitely many distinct pairs $(p^{\prime},q^{\prime})$ 
are obtained from infinitely many pairs $(p,q)$ this way if $\Psi$ tends to $0$, 
which shows $\mathscr{K}_{1}(\Psi)\subseteq \mathscr{K}_{1}^{\ast}(\Psi)$. The other inclusion is obvious.

The special case $s=1$ of Satz~6 in Jarn\'ik~\cite{jarnik}, translated into the present notation,
asserts the following.

\begin{theorem}[Jarn\'ik] \label{vjarnik}
Let $\Psi$ have the properties \eqref{eq:bed1}, \eqref{eq:bed2}.
Then
\begin{equation} \label{eq:jar}
\mathscr{K}_{1}(\Psi)\setminus{\cup_{c<1} \mathscr{K}_{1}\left(c\Psi\right)}\neq \emptyset. 
\end{equation}
\end{theorem}

In other words, for any suitable function $\Psi$, there are elements in $\mathscr{K}_{1}(\Psi)$
that do not belong to the set $\mathscr{K}_{1}(c\Psi)$ for any $c<1$. Note that
$\mathscr{K}_{1}(c\Psi)$ get larger as $c$ increases, and that \eqref{eq:jar} is stronger
than $\mathscr{K}_{1}(\Psi)\setminus{\mathscr{K}_{1}(c\Psi)}\neq \emptyset$ for all fixed $c<1$.

Condition \eqref{eq:bed1} on $\Psi$ is very natural in this context due to \eqref{eq:diri} for $k=1$,
and it seems it cannot be weakened in a reasonable way, as we briefly carry out.
Basic facts on one-dimensional Diophantine approximation show
\begin{equation} \label{eq:wurzel5}
\vert \zeta q-p\vert \leq \frac{1}{\sqrt{5}}q^{-1}
\end{equation}
has infinitely many integer solutions $p,q$, see~\cite{perron}. Thus
\eqref{eq:jar} cannot hold if $\Psi(x)>Dx^{-1}$ holds for some $D>1/\sqrt{5}$
and all large $x$, for example for $\Psi(x)=(1/2)x^{-1}$. 
Moreover, it cannot hold for $\Psi(x)=Dx^{-1}$ if $D\in{(1/\sqrt{8},1/\sqrt{5})}$ for instance, 
due to facts connected to the Lagrange spectrum, see~\cite{cusick}. 

Condition \eqref{eq:bed2} avoids some problems and most likely cannot be dropped completely.
However, we point out that we will mostly only need the weaker assumption \eqref{eq:bed3}.

The construction in the proof of Theorem~\ref{vjarnik}, in fact shows more.
Jarn\'iks proof shows that $\zeta=[b_{0};b_{1},b_{2},\cdots]$ with convergents 
$p_{n}/q_{n}=[b_{0};b_{1},\cdots ,b_{n}]$ belongs to \eqref{eq:jar}, provided that
\begin{eqnarray}
b_{n+1}\Psi(q_{n})q_{n}&>&1,\qquad n\geq 1,  \label{eq:1a}\\
 \lim_{n\to\infty} b_{n+1}\Psi(q_{n})q_{n}&=&1. \label{eq:1b}
\end{eqnarray}
These assumptions can be satisfied since $\Psi(x)=o(x^{-1})$, moreover $\lim_{n\to\infty}b_{n}=\infty$.

However, to obtain a better result in Theorem~\ref{theor} later, we apply subtle modifications.
It is easy to see that it is sufficient to have \eqref{eq:1a} for $n\geq n_{0}$.
Hence, one may choose the initial partial quotients of $\zeta=[b_{0};b_{1},b_{2},\cdots]$ 
up to any $b_{m}$ arbitrarily. Furthermore, it becomes evident that the partial quotients of
a suitable $b_{n}$ can be individually altered to $b_{n}+1$ without affecting the result.
Combining these facts, the suitable set of $\zeta$ in the difference set in \eqref{eq:jar} is uncountable in any 
real interval. Moreover, we show that we can weaken \eqref{eq:bed2} to \eqref{eq:bed3} in 
the conditions of Theorem~\ref{vjarnik}. So assume \eqref{eq:bed3}, \eqref{eq:bed1} for $\Psi$
and repeat the above construction starting from \eqref{eq:1a}, \eqref{eq:1b}.
It is evident by the construction that for given $c<1$, {\em all} 
convergents in lowest terms $p_{n}/q_{n}$ of $\zeta$ of sufficiently large index 
$n\geq \hat{n}(c)$ have the property
\begin{equation} \label{eq:schatz}
c\Psi(q_{n})<\vert q_{n}\zeta-p_{n}\vert<\Psi(q_{n}). 
\end{equation}
Moreover, it is well-known that 
for $(p,q)$ linearly independent to all $(p_{n},q_{n})$ we have $\vert q\zeta-p\vert>(1/2)q^{-1}$,
see Satz~11 in~\cite{perron}, so for large $q$ condition \eqref{eq:bed1} yields
\[
\vert q\zeta-p\vert>\frac{1}{2q}> \Psi(q)>c\Psi(q).
\]
Hence we may restrict to the case $(p,q)$ are integral multiples of $(p_{n},q_{n})$ with 
$p_{n}/q_{n}$ a convergent of $\zeta$ in lowest terms. 
However, for those pairs \eqref{eq:bed3},\eqref{eq:vervielfachen} and \eqref{eq:schatz}
imply $c\Psi(q)<\vert \zeta q-p\vert$ as well. Combination of these facts indeed
show the assertion and this argument also shows
that when we restrict to coprime pairs $(p,q)$, condition \eqref{eq:bed1} alone is sufficient. 
Summing up all the above extensions of Theorem~\ref{vjarnik}, we formulate the result as a theorem.

\begin{theorem} \label{jth}
Let $\Psi$ satisfy  {\upshape\eqref{eq:bed1}}, {\upshape\eqref{eq:bed3}}.
Let $I$ be any non-empty open real interval. Then the set
\begin{equation} \label{eq:jarn}
\left(\mathscr{K}_{1}(\Psi)\setminus{\cup_{c<1} \mathscr{K}_{1}\left(c\Psi\right)}\right)
\bigcap I
\end{equation}
is uncountable. The same is true for $\mathscr{K}_{1}^{\ast}$, where we may 
drop the condition {\upshape\eqref{eq:bed3}}.
\end{theorem}

Jarn\'ik, among others, tried to generalize this concept to simultaneous approximation.
Let $k\geq 1$ be an integer, $\Psi$ a function and 
$\underline{\zeta}=(\zeta_{1},\ldots,\zeta_{k})\in{\widehat{\mathscr{C}}_{k}\subseteq \mathbb{R}^{k}}$ 
where $\widehat{\mathscr{C}}_{k}$ denotes the set of vectors in $\mathbb{R}^{k}$ that are linearly 
independent over $\mathbb{Q}$ together with $1$. Define the set  
\[
\widehat{\mathscr{K}}_{k}(\Psi)= 
\{\underline{\zeta}\in{\widehat{\mathscr{C}}_{k}}: \max_{1\leq j\leq k}\vert \zeta_{j} q-p_{j}\vert\leq \Psi(q)
\text{ for infinitely many} \quad (q,p_{1},\ldots,p_{k})\in{\mathbb{N}\times\mathbb{Z}^{k}}\}.
\]
This definition implies that any element in $\widehat{\mathscr{K}}_{k}(\Psi)\subseteq{\mathbb{R}^{k}}$
gives rise to several elements in $\widehat{\mathscr{K}}_{l}(\Psi)\subseteq{\mathbb{R}^{l}}$ for any $1\leq l<k$,
by taking subsets.
Apriori an analogue result is not clear for any $l>k$,
since trivial extensions like 
$\underline{\zeta}=(\zeta_{1},\zeta_{2},\ldots,\zeta_{k},\zeta_{k},\ldots,\zeta_{k})\in{\mathbb{R}^{l}}$,
are excluded by the restriction to $\widehat{\mathscr{C}}_{k}$.
The general claim of Satz~6 in~\cite{jarnik} is the following.
\begin{theorem}[Jarn\'ik]  \label{jth2}
For any positive integer $k$ and any $\Psi$ satisfying 
{\upshape\eqref{eq:bed1}},{\upshape\eqref{eq:bed2}}, we have
\begin{equation} \label{eq:jj}
\widehat{\mathscr{K}}_{k}(\Psi)\setminus{\cup_{c<1} 
\widehat{\mathscr{K}}_{k}\left(c\Psi\right)}\neq \emptyset.
\end{equation}
\end{theorem}

Jarn\'ik uses induction on $k$ to infer \eqref{eq:jj} from \eqref{eq:jar}, which
in fact shows that an extension of an element in $\widehat{\mathscr{K}}_{k}(\Psi)\subseteq{\mathbb{R}^{k}}$
to elements in $\widehat{\mathscr{K}}_{l}(\Psi)\subseteq{\mathbb{R}^{l}}$ for $l>k$ indeed exists.
Observe, however, that by \eqref{eq:diri} the natural condition would be $\Psi(x)=o(x^{-1/k})$ instead
of {\upshape\eqref{eq:bed1}}.
We quote some more results in this manner, which indeed require only the decay $\Psi(x)=o(x^{-1/k})$ or
even something slightly weaker. The first is a corollary of~\cite[Satz 5]{jarnik}.

\begin{theorem}[Jarn\'ik] \label{jthm}
Let $\varphi$ and $\lambda$ be positive decreasing to $0$ functions of a positive integer argument
such that the series $\sum_{n\geq 1} \varphi(n)^{k}/n$ converges. Then there is an uncountable
family of vectors $(\zeta_{1},\ldots,\zeta_{k})\in{\widehat{\mathscr{C}}_{k}}$ such that
\[
\max_{1\leq j\leq k} \Vert q\zeta_{j}\Vert \geq \lambda(q)\cdot \varphi(q)q^{-1/k}
\]
for every integer $q>0$ and
\[
\max_{1\leq j\leq k} \Vert q\zeta_{j}\Vert \leq \varphi(q)q^{-1/k}
\]
has infinitely many integral solutions $q>0$.
\end{theorem}

This was first improved by Akhunzhanov and Moshchevitin~\cite{akhmosh} and 
recently improved further by Akhunzhanov~\cite{azhu}.

\begin{theorem}[Akhunzhanov]  \label{akh}
Let $v(k)$ be the volume of the $k$-dimensional unit ball. Put $C=2(k+1)^{1/(2k)}v(k)^{-1/k}$ and
$B=2(2k+3)(C(2k+3))^{1+1/k}$. Let $\varphi(q)$ be a decreasing function such that 
$\varphi(1)\leq (3\cdot 4^{k}C(2k+3))^{-1}$. Then there is an uncountable
family of vectors $(\zeta_{1},\ldots,\zeta_{k})\in{\widehat{\mathscr{C}}_{k}}$ such that
\[
\max_{1\leq j\leq k} \Vert q\zeta_{j}\Vert \geq (1-B\varphi(q)^{1+1/k})\cdot \varphi(q)q^{-1/k}
\]
for every integer $q>0$ and
\[
\max_{1\leq j\leq k} \Vert q\zeta_{j}\Vert \leq (1+B\varphi(q)^{1+1/k})\cdot \varphi(q)q^{-1/k}
\]
has infinitely many integral solutions $q>0$.
\end{theorem}

In particular, in our notation for any $\Psi(x)=o(x^{-1/k})$ we have 
\[
\bigcap_{c<1} \left(\widehat{\mathscr{K}}_{k}(\Psi/c)\setminus
\widehat{\mathscr{K}}_{k}\left(c\Psi\right)\right)\neq \emptyset,
\]
which is very close to \eqref{eq:jj}.
Another somehow related fact we want to quote is due to Beresnevich~\cite{beres}.

\begin{theorem}[Beresnevich]
For any integer $k\geq 1$, the set
\[
\rm{BAD_{k}}=\{\zeta\in{\mathbb{R}}: \quad \exists \gamma(\zeta)>0 \quad \text{such that} \quad 
\max_{1\leq j\leq k} \Vert q\zeta^{j}\Vert\geq \gamma(\zeta)q^{-1/k}\}
\]
has Hausdorff dimension $1$.
\end{theorem}

Concrete examples of numbers in $\rm{BAD}_{k}$ are algebraic numbers of degree $k+1$, see~\cite{peron}.

\section{Formulation of the main new results}

\subsection{Real $\zeta$}
Based on \eqref{eq:jarn},
we want to generalize \eqref{eq:jj} to simultaneous approximation of vectors of the form
$\underline{\zeta}=(\zeta,\zeta^{2},\ldots,\zeta^{k})$ with a completely different 
approach. For this reason, define 
\[
\mathscr{K}_{k}(\Psi)= \{\zeta\in{\mathscr{C}_{k}}: \max_{1\leq j\leq k}\vert \zeta^{j} q-p_{j}\vert\leq \Psi(q)
\text{ for infinitely many} \quad (q,p_{1},\ldots,p_{k})\in{\mathbb{N}\times\mathbb{Z}^{k}}\},
\]
where $\mathscr{C}_{k}\subseteq \mathbb{R}$ is defined as the real numbers not 
algebraic of degree $\leq k$.
Clearly, for any function $\Psi$ (no assumptions are required) the inclusions
\begin{equation} \label{eq:kette}
\cdots \subseteq \mathscr{K}_{3}(\Psi)\subseteq \mathscr{K}_{2}(\Psi)
\subseteq \mathscr{K}_{1}(\Psi)=\widehat{\mathscr{K}}_{1}(\Psi)
\end{equation}
hold, and any element in $\mathscr{K}_{k}(\Psi)\subseteq{\mathbb{R}}$ gives rise to some 
element in $\widehat{\mathscr{K}}_{k}(\Psi)\subseteq{\mathbb{R}^{k}}$, but the reverse is (in general) false.
Define $\mathscr{K}_{k}^{\ast}(\Psi)$ similarly to $k=1$. The above clearly holds for
$\mathscr{K}_{k}^{\ast}(\Psi)$ too, and assuming \eqref{eq:bed3} we again have
$\mathscr{K}_{k}(\Psi)= \mathscr{K}_{k}^{\ast}(\Psi)$.
As indicated, we want to extend both Theorems~\ref{jth}, \ref{jth2} for vectors of successive powers. 
It turns out that we can even weaken \eqref{eq:bed1} for $k\geq 2$.
We prove the following.

\begin{theorem} \label{theor}
Let $k\geq 2$ be an integer, the function $\Psi$ satisfy {\upshape\eqref{eq:bed3}} and $I\subseteq(-1/2,1/2)$,
$J\subseteq{\mathbb{R}}$ be non-empty open intervals.
\begin{itemize}
\item
If $\Psi$ additionally satisfies {\upshape\eqref{eq:bed4}}, then the set
\[
\left(\mathscr{K}_{k}(\Psi)\setminus{\cup_{c<1} \mathscr{K}_{k}\left(c\Psi\right)}\right)\bigcap I
\]
is uncountable. 
\item
If $\Psi$ additionally satisfies {\upshape\eqref{eq:bed5}}, for any fixed $c_{0}<1$ the set
\[
\left(\mathscr{K}_{k}(\Psi)\setminus{\mathscr{K}_{k}\left(c_{0}\Psi\right)}\right)\bigcap J
\]
is uncountable. An effective constant $d$ in {\upshape\eqref{eq:bed5}} depending only on $J$ can be given.
\end{itemize}
In both claims, the same is true for $\mathscr{K}_{k}^{\ast}$, where we may 
drop the condition {\upshape\eqref{eq:bed3}}. Elements in the sets can effectively
be determined.
\end{theorem}

\begin{remark}
As indicated before, at least the first claim is wrong for $k=1$, since \eqref{eq:wurzel5}
has arbitrarily large solutions. Recall also the counterexamples and reference subsequent 
to \eqref{eq:wurzel5}. 
\end{remark}

We will utilize results from Section~\ref{sektion1} to infer this result from Theorem~\ref{jth}
in Section~\ref{beweis}.
The effectiveness of the proof of the first part admits an interesting result concerning
functions $\Psi$ which satisfy slightly more rigid restrictions. For technical reasons,
we now assume $\Psi$ to be defined on $\mathbb{R}_{>0}$.

\begin{definition} \label{admissible}
For a positive integer $k$, say a function $\Psi:\mathbb{R}_{>0}\mapsto \mathbb{R}_{>0}$ 
is {\em admissible of degree $k$} if it satisfies 
\begin{align}
\Psi(x)&=o(x^{-2k+1}), \qquad x\to\infty.  \label{eq:admissible2}  \tag{A0.k}  \\
\frac{\Psi(x)}{\Psi(y)}&\geq\left(\frac{y}{x}\right)^{k-1}, 
\qquad  \text{for all large} \quad x\leq y.  \label{eq:admissible1}   \tag{B0.k}
\end{align}
We call $\Psi$ {\em strictly admissible} if it is admissible for any $k$.
\end{definition}

Observe that the admissibility condition becomes stronger as $k$ increases,
and for $k=1$ it is equivalent to \eqref{eq:bed1}, \eqref{eq:bed2}.
More precisely, 
\begin{align*}
 \cdots\Longrightarrow(A0.2)\Longrightarrow(A0.1)&=\eqref{eq:bed1} \Longrightarrow \eqref{eq:bed5}
\Longrightarrow \eqref{eq:bed4} ,  \\
 \cdots\Longrightarrow(B0.2)\Longrightarrow(B0.1)&= \eqref{eq:bed2}
\Longrightarrow \eqref{eq:bed3}.
\end{align*}
We give some examples of admissible functions. Let $c>0, \epsilon>0, \sigma>0$ arbitrary generically.
Any function $\Psi(x)=cx^{-2k+1-\epsilon}$
is admissible of degree $k$. More general, any map $\Psi(x)=x^{-2k+1}\varphi(x)$
is admissible of degree $k$ for any function $\varphi:\mathbb{R}_{>0}\mapsto \mathbb{R}_{>0}$ 
which tends to $0$ monotonically for large $x$.
Any map $\Psi(x)=c\cdot\rm{exp}(-\epsilon x)$ is
strictly admissible, which is equivalent to the the fact that $x\mapsto c\cdot\rm{exp}(\epsilon x)x^{-k}$ 
increases for $x\geq x_{0}$. More general, any map $\Psi(x)=c\cdot \rm{exp}(-\epsilon x^{\sigma})$
is strictly admissible.

\begin{theorem} \label{satz}
Let $\Psi$ be admissible of degree $k$. 
Define a sequence of functions $\Psi_{1}=\Psi,\Psi_{2},\Psi_{3},\ldots$ by
\[
\Psi_{j}(x)=\Psi(x^{1/j})x^{(j-1)/j}, \qquad j\geq 1.
\]
Then the functions $\Psi_{1},\ldots,\Psi_{k}$ satisfy {\upshape\eqref{eq:bed1}}, {\upshape\eqref{eq:bed2}}
and the sets
\begin{equation} \label{eq:lio}
\mathscr{K}^{j}:= \mathscr{K}_{j}(\Psi_{j})\setminus
{\cup_{c<1} \mathscr{K}_{j}\left(c\Psi_{j}\right)}, \qquad 1\leq j\leq k
\end{equation}
coincide within $(-1/2,1/2)$, i.e. $(-1/2,1/2)\cap \mathscr{K}^{j}=(-1/2,1/2)\cap \mathscr{K}^{1}$ for
all $j\geq 1$.
If $\Psi$ is even strictly admissible, then all $\Psi_{j}$ are
strictly admissible as well and the sets $\mathscr{K}^{j}$
coincide for all $j\geq 1$.

All of this is true for $\mathscr{K}^{j \ast}$ (defined similarly) too, where we may 
drop the condition {\upshape\eqref{eq:admissible1}} on $\Psi$.
\end{theorem}

\begin{remark}
It might be possible to relax the conditions on $\Psi$ to 
obtain functions $\Psi_{j}$ with weaker conditions,
similarly to Theorem~\ref{jth}, that still satisfy the claims of the theorem. 
We will not deal with this question. 
\end{remark}

\begin{corollary} \label{corkor}
Let $\Psi$ be admissible of degree $k$ or strictly admissible respectively,
and $I\subseteq{(-1/2,1/2)}$ a non-empty open interval. 
Then the functions $\Psi_{1}=\Psi,\Psi_{2},\Psi_{3},\ldots$ defined in Theorem~{\upshape\ref{satz}} 
have the property that the sets
\[
\bigcap_{1\leq j\leq k} \left(\mathscr{K}_{j}(\Psi_{j})\setminus
{\cup_{c<1} \mathscr{K}_{j}\left(c\Psi_{j}\right)}\right)\cap I, \qquad 
\bigcap_{j\geq 1} \left(\mathscr{K}_{j}(\Psi_{j})\setminus
{\cup_{c<1} \mathscr{K}_{j}\left(c\Psi_{j}\right)}\right)\cap I
\]
respectively, are uncountable.
\end{corollary}

\begin{proof}
Combination of Theorem~\ref{satz} and Theorem~\ref{jth}.
\end{proof}

Roughly speaking, the theorem tells us that under moderate assumptions on $\Psi$,
if $\zeta\in{\mathbb{R}}$ can be approximated to some degree $\Psi$ and no better, 
then its powers $\zeta,\zeta^{2},\ldots,\zeta^{k}$
can be simultaneously approximated to some modified degree $\Psi_{k}$,
that can effectively be determined, and no better. 

\subsection{Fractal sets} \label{fractalgenial}

We turn towards approximation of Cantor set type numbers by rationals.
Recall the Cantor set can be defined as the numbers in $[0,1]$ that allow a representation
\[
c_{1}3^{-1}+c_{2}3^{-2}+c_{3}3^{-3}+\cdots, \qquad c_{i}\in{\{0,2\}}.
\]
So apart from special rational numbers with denominator a power of 3,
whose ternary representation is not unique, it coincides with the numbers
who have no $1$ in the unique ternary representation. For sets with
similar missing digit properties, Bugeaud's~Theorem~1 in~\cite{buge}, 
whose proof originates in a special form of the Folding Lemma~\cite{vdp}, contributes the following. 

\begin{theorem}[Bugeaud] \label{bu}
For an integer $b\geq 2$, let $J(b)\subseteq\{0,1,\ldots,b-1\}$ with at least two elements.
Denote by $K_{J(b)}$ the numbers in $[0,1]$ whose base $b$ expansion contains 
only digits in $J(b)$. Let $\Psi$ satisfy {\upshape\eqref{eq:bed1}}, {\upshape\eqref{eq:bed2}}. 
Then for any $c<1/b$, the set
\[
\left(\mathscr{K}_{1}(\Psi)\setminus{\mathscr{K}_{1}\left(c\Psi\right)}\right)\bigcap K_{J(b)}
\]
is uncountable.
\end{theorem}

It is not obvious that the proof of Theorem~\ref{bu} can be modified in the way 
Theorem~\ref{vjarnik} was modified to obtain Theorem~\ref{jth}, to deduce
\eqref{eq:bed2} can be weakened to \eqref{eq:bed3}. For technical reasons we will restrict to the case $J(b)=\{0,1\}$, 
however we point out the results should remain true in general, but the
proofs become more technical in several ways.
Observe that a general element $\zeta$, restricted by an arbitrary digit set
$J(b)$ with two elements, can be derived from an element
in the special set where $J(b)=\{0,1\}$ by a transformation $\zeta\mapsto A\zeta+B$, where
$A\in{\{1,2,\ldots,b-1\}}$ and $B=s(b^{-1}+b^{-2}+\cdots)=s/(b-1)$ for $s\in{\{0,1,2,\ldots,b-2\}}$. 
The assumption that $J(b)$ has precisely two elements is no restriction as
for $\vert J(b)\vert>2$ our claims will follow trivially from the case $\vert J(b)\vert=2$.

First we consider the case $k\geq 2$ and turn to $k=1$ later.
Thanks to Corollary~\ref{lemma}, the case $k\geq 2$ will be indeed
easier.
Proceeding very similar to the proof of (the first assertion of) Theorem~\ref{theor} 
yields the following. 

\begin{theorem} \label{jox}
Let $k\geq 2, b\geq 2$ be integers and $K_{J(b)}$ be as in Theorem~{\upshape\ref{bu}} with $J(b)=\{0,1\}$. 
Assume the function $\Psi$ satisfies {\upshape\eqref{eq:bed4}}, {\upshape\eqref{eq:bed3}}. 
For any $c<1/b$, the set
\[
\left(\mathscr{K}_{k}(\Psi)\setminus{\mathscr{K}_{k}\left(c\Psi\right)}\right)\bigcap
K_{J(b)}
\]
is uncountable. The same holds for $\mathscr{K}_{k}^{\ast}$,
where we may drop the condition {\upshape\eqref{eq:bed3}}.
\end{theorem}

However, we can do slightly better. Using an approach similar to
Theorem~1.26 in~\cite{schlei}, with slight refinements we will establish in Section~\ref{prefra}, we can 
essentially improve the bound $1/b$ to $1/(b-1)$. 

\begin{theorem} \label{begeisterung}
Let $k\geq 2, b\geq 3$ be integers and $K_{J(b)}$ be as in Theorem~{\upshape\ref{bu}} with $J(b)=\{0,1\}$. 
Assume the function $\Psi$ satisfies {\upshape\eqref{eq:bed4}}, {\upshape\eqref{eq:bed3}}.
Then the set
\begin{equation}  \label{eq:fraktal}
\left(\mathscr{K}_{k}(\Psi)\setminus{\cup_{c<\frac{1}{b-1}}\mathscr{K}_{k}\left(c\Psi\right)}\right)\bigcap
K_{J(b)}
\end{equation}
is uncountable. The same holds for $\mathscr{K}_{k}^{\ast}$,
where we may drop the condition {\upshape\eqref{eq:bed3}}. Elements in \eqref{eq:fraktal} can be
effectively constructed.
\end{theorem}

We return to $k=1$. Let $\gamma=(1+\sqrt{5})/2\approx 1.6180$ be the golden ratio.
If we restrict to functions $\Psi$ with the stronger decay condition 
\begin{equation}
\Psi(x)<x^{-\gamma-\epsilon}, \qquad \text{for some (arbitrarily small)} 
\quad \epsilon>0,\quad \text{for all} \quad x\geq \hat{x},   \label{eq:staerker}  \tag{$A^{\prime}$}
\end{equation} 
the result of Theorem~\ref{begeisterung} can be extended to $k=1$.

\begin{theorem}  \label{kgleicheins}
Let $b\geq 3$ be an integer and $K_{J(b)}$ be as in Theorem~{\upshape\ref{bu}} with $J(b)=\{0,1\}$. 
Assume the function $\Psi$ satisfies {\upshape\eqref{eq:staerker}}, {\upshape\eqref{eq:bed3}}.
Then the set
\begin{equation}  \label{eq:fraktal2}
\left(\mathscr{K}_{1}(\Psi)\setminus{\cup_{c<\frac{1}{b-1}}\mathscr{K}_{1}\left(c\Psi\right)}\right)\bigcap
K_{J(b)}
\end{equation}
is uncountable. The same holds for $\mathscr{K}_{1}^{\ast}$,
where we may drop the condition {\upshape\eqref{eq:bed3}}.
Elements in \eqref{eq:fraktal2} can be effectively constructed.
\end{theorem}

\begin{remark}
Observe that a density result in the spirit of Theorem~\ref{theor} 
cannot hold for the set in \eqref{eq:fraktal} or \eqref{eq:fraktal2} for $b\geq 3$ by definition of $K_{J(b)}$.
\end{remark}

We provide several more remarks to the results of the current Section~\ref{fractalgenial} in Section~\ref{vier}.

\subsection{Consequences and the relation to known results}
Before we turn to the proofs, we want to discuss the assertion of Theorem~\ref{theor}.
Assume $k\geq 2$ fixed. As pointed out preceding Lemma~\ref{jthm},
estimate \eqref{eq:diri}
for the case $\underline{\zeta}\in{\mathbb{R}^{k}}$ linearly independent together with $1$,
suggest that Theorem~\ref{theor} might hold under the weaker condition 
\begin{equation} \label{eq:abed}
\Psi(x)=o(x^{-\frac{1}{k}}).   \tag{$A^{\ast}$}
\end{equation}
Theorems~\ref{jth2},~\ref{jthm} and~\ref{akh} are also affirmative.
However, for functions that do not satisfy any of the (A)-conditions from Section~\ref{formulation}, 
even much weaker claims are unknown. Bugeaud and Laurent \cite{buglau} introduced the exponent 
$\lambda_{k}(\zeta)$ as the supremum of real $\eta$ such that
\[
\max_{1\leq j\leq k} \vert \zeta^{j}q-p_{j}\vert \leq q^{-\eta} 
\]
has infinitely many solutions $(q,p_{1},\ldots,p_{k})\in{\mathbb{N}\times \mathbb{Z}^{k}}$. 
Denote $\rm{Spec}(\lambda_{k})$ the spectrum of $\lambda_{k}(\zeta)$
as $\zeta$ runs through all real numbers not algebraic of degree $\leq k$. 
By virtue of \eqref{eq:diri} we have $\rm{Spec}(\lambda_{k})\subseteq[1/k,\infty]$.
For $k\geq 3$,
it is still unknown if actually $\rm{Spec}(\lambda_{k})=[1/k,\infty]$, which was posed in~\cite[Problem~1.3]{bug}.
Clearly, a result in the spirit of Theorem~\ref{theor} under the weaker condition 
\eqref{eq:abed} would imply a positive answer on the spectrum problem, but the
reverse implication is far from being true. 
For $k=2$, a positive answer to the spectrum problem was established in~\cite{bere},~\cite{vel}
by metrical arguments. However they do not allow for deducing that the set 
$\mathscr{K}_{k}(\Psi)\setminus \mathscr{K}_{k}(c\Psi)$ is non-empty
for $\Psi(x)=x^{-\nu}$ with $\nu>1/2$ and any $c\in{(0,1)}$. 

Let us return to the case of $\Psi$ that satisfies an (A)-condition
from Section~\ref{formulation}. An explicit construction
leading to $\rm{Spec}(\lambda_{k})\supseteq[1,\infty]$ was given in~Theorem~2 in~\cite{bug}. 
An explicit construction of $\zeta$ with prescribed exponent $\lambda_{k}(\zeta)\geq 1$
with the additional property that $\zeta$ belongs the Cantor set or similar fractal sets,
was established by the author in \cite{schlei}, improving a slightly weaker result from~\cite{buge}.
However, for $k\geq 2$, again no explicit constant $c>0$ for which
$\mathscr{K}_{k}(\Psi)\setminus \mathscr{K}_{k}(c\Psi)\neq \emptyset$
holds with $\Psi(x):=x^{-\nu}$ for any $\nu\geq 1$ has been known. 
Theorem~\ref{theor} is a satisfactory result for functions $\Psi(x)$ with decay condition
\eqref{eq:bed1}, or actually slightly weaker.

\section{Preparatory results}  \label{sektion1}

\subsection{Preparatory results for the general case}
The proofs in Section~\ref{beweis} and Section~\ref{vier} will rely heavily on the following elementary observation.

\begin{lemma} \label{lemma2}
Let $k\geq 2$ be an integer and $\zeta$ a real number.
Suppose $\vert \zeta-p/q\vert=dq^{-k}$ with $d<1/2$ and integers $p,q$ with $q$ sufficiently large holds
(then $p$ is large too). In case of $\zeta\in{(0,1/2)}$, we have
\begin{equation} \label{eq:sieben}
\max_{1\leq j\leq k} \Vert q^{k}\zeta^{j} \Vert=q^{k}\left\vert \zeta-\frac{p}{q}\right\vert=
q^{k-1}\vert q\zeta-p\vert= d.
\end{equation}
In any case, we have
\begin{equation} \label{eq:viel}
\max_{1\leq j\leq k} \Vert q^{k}\zeta^{j} \Vert=\vert L_{k}+o(1)\vert\cdot q^{k}\left\vert \zeta-\frac{p}{q}\right\vert
=\vert L_{k}+o(1)\vert\cdot q^{k-1}\left\vert q\zeta-p\right\vert,
\end{equation}
where $L_{k}=L_{k}(\zeta):=\max_{1\leq j\leq k} (j\zeta^{j-1})$, as $q\to\infty$. 
\end{lemma}

\begin{proof}
The condition $\vert \zeta-p/q\vert=dq^{-k}$ with $d<1/2$ implies
\begin{equation} \label{eq:erster}
\vert\zeta q^{k}-pq^{k-1}\vert=d<1/2.
\end{equation}
In particular $pq^{k-1}$ is the closest integer to $\zeta q^{k}$.
More general, by the assumption $\zeta<1/2$ and since $\vert \zeta-p/q\vert$ 
is very small by assumption, for large $q$ we also have $0<p/q<1/2$.
It is easy to check that
\begin{equation} \label{eq:einstprop}
j(1/2)^{j-1}\leq 1, \qquad \qquad j\in{\{1,2,3,\ldots\}}.
\end{equation}
This implies
\begin{equation} \label{eq:netz}
 \left\vert \zeta^{j-1}+\cdots+\frac{p^{j-1}}{q^{j-1}}\right\vert<j(1/2)^{j-1}\leq 1, \qquad 2\leq j\leq k
\end{equation} 
for large $q$. Hence the calculation
\[
\vert q^{k}\zeta^{j}-p^{j}q^{k-j}\vert=q^{k}\left\vert \zeta^{j}-\frac{p^{j}}{q^{j}}\right\vert
=q^{k}\left\vert \zeta-\frac{p}{q}\right\vert \left\vert \zeta^{j-1}+\cdots+\frac{p^{j-1}}{q^{j-1}}\right\vert,
\quad 2\leq j\leq k
\]
and \eqref{eq:erster} shows that $p^{j}q^{k-j}$ is the closest integer to $q^{k}\zeta^{j}$ for $1\leq j\leq k$ and
furthermore the maximum of $\Vert \zeta^{j}q^{k}\Vert$ among $j\in{\{1,2,\ldots,k\}}$ is obtained for $j=1$.
Thus \eqref{eq:erster} indeed proves \eqref{eq:sieben}. For \eqref{eq:viel} one can proceed
very similarly, using that the left hand side of \eqref{eq:netz} tends to $j\zeta^{j-1}$ as $q\to\infty$
since $p/q$ tends to $\zeta$.
\end{proof}

\begin{remark}
Due to Satz~11 in~\cite{perron} mentioned already preceding Theorem~\ref{jth},
the assumption of Lemma~\ref{lemma2}
implies $p/q$ must be a convergent of the continued fraction expansion of $\zeta$.
\end{remark}

\begin{remark}
The remainder term in \eqref{eq:viel} can be estimated in dependence of $q$.
\end{remark}

We look at the values $L_{k}(\zeta)$ more closely. The following proposition
comprises the most important properties of this quantity and will be helpful
particularly in the proof of the second assertion of Theorem~\ref{theor}.

\begin{proposition} \label{eigenschaften}
Let $k\geq 2$ be an integer. Consider $L_{k}=L_{k}(\zeta)$ from 
Lemma~{\upshape\ref{lemma2}} as a function of $\zeta\in{\mathbb{R}_{>0}}$. 
Then $L_{k}$ is continuous, has image $[1,\infty)$, is constant 
$L_{k}(\zeta)=1$ in $(0,1/2]$ and strictly increasing in $(1/2,\infty)$.
\end{proposition}

\begin{proof}
Apart from $j=1$, any expression $j\zeta^{j-1}$ involved in the maximum is continuous,
strictly increasing and tends to infinity as a function of $\zeta$. 
It follows that the maximum $L_{k}=L_{k}(\zeta)$ is continuous, non-decreasing
and strictly increases unless
it is obtained for $j=1$, which in view of \eqref{eq:einstprop}
is easily seen to be equivalent to $L_{k}=1$ and $\zeta\in{(0,1/2]}$.
\end{proof}

We quote Lemma~2.4 and Corollary~3.1 in~\cite{schlei} in slightly
modified versions, such as the following additional results. 
To Lemma~2.4 in~\cite{schlei} we add the result \eqref{eq:aufgfaedelt} which was
inferred within its proof in~\cite{schlei} but not explicitly mentioned.
Furthermore, in view of Proposition~\ref{eigenschaften},
we can improve the original bound $C_{0}$ from Lemma~2.4 
similarly to the proof of Lemma~\ref{lemma2} to any constant smaller than $(1/2)L_{k}(\zeta)^{-1}$,
when restricting to sufficiently large integers only. Furthermore in case of $\zeta\in{(0,1/2)}$ the constant
can be put $1/2$. This was already pointed out in~\cite[Remark~2.4]{schlei}.

\begin{lemma}[S., 2014] \label{lemmaalt}
Let $k$ be a positive integer and $\zeta$ be a positive real number.
For an integer $z$ and $1\leq j\leq k$, denote by $y_{j}$ the closest integer to $\zeta^{j}z$.
There exists a constant $C=C(k,\zeta)>0$ such that for any large integer $z\geq \hat{z}>0$ the estimate 
\begin{equation} \label{eq:hain}
\max_{1\leq j\leq k}\Vert \zeta^{j} z\Vert < C\cdot z^{-1},
\end{equation}
implies $y_{1}/z=y_{0}/z_{0}$ for integers $(z_{0},y_{0})=1$ and $z_{0}^{k}$ divides $z$.
A suitable choice for $C$ is given by $C=1/2$ if $\zeta\in{(0,1/2)}$ and
$C=C_{0}:=(1/2)L_{k}(\zeta)^{-1}-\epsilon$ with $L_{k}$
from Lemma~{\upshape\ref{lemma2}} and arbitrary $\epsilon>0$ (where $\hat{z}$ above depends on $\epsilon$).

Moreover, $y_{0}^{j}/z_{0}^{j}$ is a convergent of the continued fraction expansion of $\zeta^{j}$
for $1\leq j\leq k$.
Furthermore, provided {\upshape\eqref{eq:hain}} holds for some pair $(z,\widetilde{C})$, then
it holds for any pair $(z^{\prime},\widetilde{C})$ with $z^{\prime}$ a
positive integral multiple of $z_{0}^{k}$ not larger than $z$, and the best possible value $C$ 
in {\upshape\eqref{eq:hain}} is obtained for $z^{\prime}=z_{0}^{k}$. More precisely,
we have 
\begin{equation} \label{eq:aufgfaedelt}
(z,y_{1},\ldots,y_{k})~=~M\cdot~(z_{0}^{k},z_{0}^{k-1}y_{0},\ldots,y_{0}^{k})
\end{equation}
for some positive integer $M$ for any solution of {\upshape\eqref{eq:hain}}.
\end{lemma}

\begin{corollary}[S., 2014]  \label{korollar}
Let $k\geq 2$ be an integer, $\zeta$ be a real number. For any fixed $T>1$,
there exists $\hat{z}=\hat{z}(T,\zeta)$, such that the estimate
\[
\max_{1\leq j\leq k}\Vert \zeta^{j} z\Vert \leq z^{-T}
\]
for an integer $z\geq \hat{z}$ implies that for $z_{0},y_{0}$ as in Lemma~{\upshape\ref{lemmaalt}} 
we have 
\begin{equation}  \label{eq:allesgilt}
\vert \zeta z_{0}-y_{0}\vert\leq z_{0}^{-kT-k+1}.
\end{equation}
Similarly, if for $C_{0}=C_{0}(k,\zeta)$ from Lemma~{\upshape\ref{lemmaalt}} the inequality
\[
\max_{1\leq j\leq k}\Vert \zeta^{j} z\Vert < C_{0}\cdot z^{-1}
\]
has an integer solution $z>0$, then {\upshape\eqref{eq:allesgilt}} holds with $T=1$. 
\end{corollary} 

We will only need special aspects of Lemma~\ref{lemmaalt} and Corollary~\ref{korollar}
for the proofs in Section~\ref{beweis}, which will be summarized in the following Corollary~\ref{lemma}
in a way that allows convenient quotation.
Concretely, Lemma~\ref{lemmaalt} in combination with the $T=1$ case
of Corollary~\ref{korollar} yield the following.

\begin{corollary} \label{lemma}
Let $k\geq 2$ be an integer and $\zeta$ be real. Define $L_{k}$ as in Lemma~{\upshape\ref{lemma2}}. 
Further let the function $\Psi$ satisfy {\upshape\eqref{eq:bed3}} and additionally either
condition \eqref{eq:bed4} for $\zeta\in{(0,1/2)}$ or
\eqref{eq:bed5} with $d=(1/2)L_{k}(\zeta)^{-1}-\epsilon$ for arbitrary $\epsilon>0$ 
in case of arbitrary $\zeta$.
All solutions $(z,y_{1},\ldots,y_{k})\in{\mathbb{N}\times\mathbb{Z}^{k}}$ of
\begin{equation} \label{eq:bra}
\max_{1\leq j\leq k}\vert \zeta^{j} z-y_{j} \vert\leq \Psi(z)
\end{equation}
with large $z$, are integral multiples of solutions of the form
$(z_{0}^{k},z_{0}^{k-1}y_{0},\ldots,y_{0}^{k})$ with $z_{0},y_{0}$ in Lemma~{\upshape\ref{lemmaalt}}.
In particular, if  {\upshape\eqref{eq:bra}} has arbitrarily large solutions,
then it has solutions with the additional property $z=q^{k}$
for $q\in{\mathbb{Z}}$ with the property $\Vert \zeta q\Vert \leq q^{-2k+1}$. 
\end{corollary}

\begin{proof}
The conditions \eqref{eq:bed5} respectively \eqref{eq:bed4} are chosen such that the assumptions 
of Lemma~\ref{lemmaalt} and Corollary~\ref{korollar} for $T=1$ are satisfied
for large $z$ for which \eqref{eq:bra} holds. 
Due to \eqref{eq:aufgfaedelt}, similarly to \eqref{eq:vervielfachen}, we have 
\[
\max_{1\leq j\leq k} \vert\zeta^{j}z-y_{j}\vert=
M \max_{1\leq j\leq k}\vert \zeta^{j}z_{0}^{k}-z_{0}^{k-j}y_{0}^{j}\vert,  
\]
so in view of \eqref{eq:bed3} indeed a solution of \eqref{eq:bra}
leads to a primitive solution of the claimed form.
Corollary~\ref{korollar} yields $\vert \zeta z_{0}-y_{0}\vert \leq z_{0}^{-2k+1}$. 
It remains to put $q=z_{0}$ in view of \eqref{eq:aufgfaedelt}.
\end{proof}

\subsection{Preparatory results for fractal sets} \label{prefra}

For the proofs concerning Section~\ref{fractalgenial} we will apply a generalization of
Lemma~4.7 in \cite{schlei}, which can be equivalently stated in the following way.

\begin{lemma}[S., 2014]  \label{minkow}
Let $k\geq 2, b\geq 2$ be integers, and $\zeta=\sum_{n\geq 1} b^{-a_{n}}$ 
for some positive integer sequence $(a_{n})_{n\geq 1}$ that satisfies $\lim_{n\to\infty} a_{n+1}/a_{n}>k$.
Then for $(x,y)\in{\mathbb{Z}^{2}}$ with sufficiently large $x$, the estimate
\[
\vert\zeta x-y\vert\leq x^{-\frac{k}{k-1}}
\]
implies $(x,y)$ is linearly dependent to some
\[
\underline{x}_{n}:=(x_{n},y_{n}):=(b^{a_{n}},\sum_{j\leq n} b^{a_{n}-a_{j}}).
\]
\end{lemma}

We extend this lemma in several ways. Firstly, we replace 
the limes condition by a more general parametric condition that involves the limes inferior.
Secondly, we deal with more general numbers $\zeta$. 
Assume the digits $1$ in base $b$ of $\zeta$ are no longer isolated, 
but there are integer blocks $I_{1},I_{2},\ldots$
where one has free choice of digits within $I=\cup I_{n}$ and everywhere else
we put the digit $0$. Provided the block lengths
are sufficiently small compared to the gaps between the blocks, a result analogue to Lemma~\ref{minkow} still holds. 
Exponential gaps between the blocks- as in Lemma~\ref{minkow} where each block consists of only one element-
and subexponential block lengths are sufficient conditions. For a rigorous definition of the blocks
see Lemma~\ref{dieprop}. 
Finally, we formulate the lemma in a slightly more general way than it will be needed, 
by introducing a parameter $A$ which in our concrete applications of Theorem~\ref{begeisterung} 
and Theorem~\ref{kgleicheins} will be just $A=1$.
However, the proofs do not become significantly more difficult in this more general context
and towards the question of more general digit sets $J(b)$ it makes sense to consider the more
general forms.
For the proof of Lemma~\ref{dieprop} we recall Proposition~4.6 from \cite{schlei}
which is in fact a special case of Minkoswki's second lattice point theorem.

\begin{proposition}[S., 2014] \label{cor}
Let $\zeta\in{\mathbb{R}}$. Then for no parameter $Q>0$ the system
\begin{equation} \label{eq:lindep}
\vert M\vert \leq Q, \qquad  \vert \zeta M- N\vert< \frac{1}{2}Q^{-1}
\end{equation}
has two linearly independent solutions $(M,N)\in{\mathbb{Z}^{2}}$.
\end{proposition}

\begin{lemma} \label{dieprop}
Let $b\geq 2$ be an integer and $A\in{\{1,2,\ldots,b-1\}}$.
Let $(e_{n})_{n\geq 1}$ and $(f_{n})_{n\geq 1}$ be strictly increasing sequences of positive integers
such that 
\begin{itemize}
\item $e_{1}<f_{1}<e_{2}<f_{2}<\cdots$,
\item  $\omega:=\liminf_{n\to\infty} e_{n+1}/f_{n}>2$, 
\item  $f_{n}-e_{n}=O(n)$ as $n\to\infty$.
\end{itemize}
Let $I_{n}=\{e_{n},e_{n}+1,\ldots,f_{n}\}$ and $I=\cup_{n\geq 1} I_{n}$. 
Let $(a_{n})_{n\geq 1}$ be a strictly increasing sequence 
of integers such that 

\begin{itemize}
\item all $a_{n}\in{I}$ 
\item the sequences $(e_{n})_{n\geq 1}$ and $(f_{n})_{n\geq 1}$ are subsequences of 
$(a_{n})_{n\geq 1}$. 
\end{itemize}

Let $\zeta=\sum_{n\geq 1} Ab^{-a_{n}}$ and for every $n\geq 1$ let $m=m(n)$ be the 
index such that $f_{n}=a_{m}$. Define
\[
\underline{x}_{n}:=(x_{n},y_{n}):=(b^{f_{n}},\sum_{j\leq m} Ab^{f_{n}-a_{j}}),
\]
such that $y_{n}=\lfloor \zeta x_{n}\rfloor$.
Let  $\epsilon>0$. Then if $n\geq \hat{n}(\epsilon)$ is sufficiently large we have
\begin{equation} \label{eq:klar}
\Vert \zeta x_{n}\Vert =\vert \zeta x_{n}-y_{n}\vert< x_{n}^{-\omega+1+\epsilon}.
\end{equation}
Moreover, for $(x,y)\in{\mathbb{N}\times \mathbb{Z}}$ with sufficiently large 
$x\geq x_{0}(\epsilon)$, the estimate
\begin{equation} \label{eq:zeus}
\vert\zeta x-y\vert\leq x^{-\frac{\omega}{\omega-1}-\epsilon}
\end{equation} 
implies $(x,y)$ is linearly dependent to some $\underline{x}_{n}$.
\end{lemma}

\begin{proof}
The first claim \eqref{eq:klar} is a straight-forward calculation using that by assumption
$(e_{n+1}-f_{n})/f_{n}>\omega-1-\epsilon$ for large $n$.
Assume the second claim is false. Then there exist $\epsilon>0$ and arbitrarily large
$(x,y)$ for which \eqref{eq:zeus} holds and which are linearly independent to all $\underline{x}_{n}$. 
Let $\delta>0$ not be too large to be specified later. 
Say $n$ is the index (large) with $b^{f_{n}}\leq x < b^{f_{n+1}}$.

First suppose $x\leq b^{f_{n+1}-(1+\delta)f_{n}}$. Put $Q=x$.
If $n$ or equivalently $x$ is sufficiently large, then by assumption \eqref{eq:zeus}
we have
\[
-\frac{\log \vert \zeta x-y\vert}{\log Q}
= -\frac{\log \vert \zeta x-y\vert}{\log x}
\geq \frac{\omega}{\omega-1}> 1.
\]
In particular for large $n$ 
\begin{equation} \label{eq:qinvers}
\vert \zeta x-y\vert<\frac{1}{2}x^{-1}=\frac{1}{2}Q^{-1}.
\end{equation}
On the other hand if $m=m(n)$ is the index such that $a_{m}=f_{n}$ then 
we infer the estimate
\begin{equation} \label{eq:trick}
\vert x_{n}\zeta-y_{n}\vert=\Vert b^{f_{n}}\zeta\Vert = \sum_{i\geq m+1} Ab^{f_{n}-a_{i}}
\leq 2A\cdot b^{f_{n}-a_{m+1}}= 2A\cdot b^{f_{n}-e_{n+1}}.
\end{equation}
For all large $n$ we now show
\begin{equation} \label{eq:glks}
2A\cdot b^{f_{n}-e_{n+1}}< \frac{1}{2}b^{(1+\delta)f_{n}-f_{n+1}}\leq \frac{1}{2}Q^{-1}.
\end{equation}
Obviously \eqref{eq:glks} is equivalent to 
\[
4A\cdot b^{f_{n+1}-e_{n+1}}\leq b^{\delta f_{n}},
\]
so it suffices to be shown that $\delta f_{n}-(f_{n+1}-e_{n+1})$ tends to infinity.
We show this is true. On the one hand $f_{n+1}-e_{n+1}=O(n)$ by assumption, but on the other hand
by the assumptions on the sequences $(e_{n})_{n\geq 1}$ and $(f_{n})_{n\geq 1}$ they 
eventually grow exponentially, such that $f_{n}/n$ clearly tends
to infinity. Thus $\delta f_{n}-(f_{n+1}-e_{n+1})$ indeed tends to infinity and consequently \eqref{eq:glks} holds.
Together with \eqref{eq:trick} this leads to 
\begin{equation} \label{eq:qinvers2}
\vert x_{n}\zeta-y_{n}\vert<\frac{1}{2}Q^{-1}.
\end{equation}
The conditions \eqref{eq:qinvers}, \eqref{eq:qinvers2} can be interpreted that for $n$ sufficiently large the system
\eqref{eq:lindep} has two linearly independent integral solutions $(M,N)=(x,y),(M,N)=(x_{n},y_{n})$. 
This leads to a contradiction to Proposition~\ref{cor}.

In the remaining case $b^{f_{n+1}-(1+\delta)f_{n}}\leq x<b^{f_{n+1}}$, put $Q=b^{f_{n+1}}$.
Let $\rho>0$ sufficiently small such that still we have $\omega-2\rho\geq 2$, for example $\rho=(\omega-2)/2$.
By assumption for sufficiently large $n$ we have $e_{n+1}/f_{n}>\omega-\rho \geq 2+\rho$.
Hence the essential argument \eqref{eq:trick}, with $n$ replaced by $n+1$, shows that
\[
\vert \zeta x_{n+1}-y_{n+1}\vert<2A\cdot Q^{1-\omega+\rho}\leq 2A\cdot Q^{-1-\rho}.
\]
In particular \eqref{eq:lindep} is satisfied for $(M,N)=(x_{n+1},y_{n+1})$ for large $n$.
On the other hand, \eqref{eq:zeus} yields
\begin{equation} \label{eq:frie}
-\frac{\log \vert \zeta x-y\vert}{\log Q}= -\frac{\log \vert \zeta x-y\vert}{\log x}\cdot\frac{\log x}{\log Q}
\geq \left(\frac{\omega}{\omega-1}+\epsilon\right)\cdot \frac{f_{n+1}-(1+\delta)f_{n}}{f_{n+1}}.
\end{equation}
The conditions on the sequences, in particular $\lim_{n\to\infty} f_{n}/n=\infty$ used already above,
imply $\liminf f_{n+1}/f_{n}=\liminf e_{n+1}/f_{n}=\omega>2$. Thus for arbitrarily small
$\eta>0$ and all sufficiently large $n\geq n_{0}(\eta)$ the most right factor in \eqref{eq:frie} can be estimated by
\[
\frac{f_{n+1}-(1+\delta)f_{n}}{f_{n+1}}\geq \frac{\omega-\eta-1-\delta}{\omega-\eta}.
\]
As $\eta$ tends to $0$ the right hand side converges to $(\omega-1)/\omega+\delta/\omega$.
Recall we may choose $\delta$ arbitrarily small. Since $\omega$ and $\epsilon>0$ are fixed,
a suitable choice of $\delta$ inserted in the right hand side of \eqref{eq:frie}
yields a lower bound $1+\tau$ for some $\tau>0$ uniformly for all large $n$. 
Consequently this bound is valid for the left hand of \eqref{eq:frie} too. We again conclude that
the system \eqref{eq:lindep} has the solution $(M,N)=(x,y)$ for sufficiently large $n$ (or $Q$), 
such that the system \eqref{eq:lindep} has two linearly independent integral solution pairs 
$(x,y),(x_{n+1},y_{n+1})$, contradiction to Proposition~\ref{cor}.
\end{proof}

\begin{remark}
The proof shows that the condition $f_{n}-e_{n}=o(2^{n})$ is sufficient.
\end{remark}

\begin{remark} \label{feeertig}
We point out that the bound in~\eqref{eq:zeus} is certainly not optimal.
Using the variant of the Folding~Lemma in~\cite{buge} one can readily improve
the exponent in \eqref{eq:zeus}, which directly translates into weaker conditions on 
$\Psi$ in Theorem~\ref{kgleicheins}. 
However, Lemma~\ref{dieprop} is false for some exponent smaller than $-1$ in \eqref{eq:zeus}.
\end{remark}

The block construction in Lemma~\ref{dieprop}
allows more flexibility for the considered numbers $\zeta$
which will lead to the improvement from $c<1/b$ to $c<1/(b-1)$ in Theorem~\ref{begeisterung}. 
However, the main observation in order to improve the bound from $1/b$ to $1/(b-1)$
in Theorem~\ref{begeisterung} is the following proposition, which gives
an estimate how well real numbers in given intervals can be approximated
by special numbers in $K_{J(b)}$ whose base $b$ digits in certain blocks can 
be chosen freely within $J(b)$, in accordance with Lemma~\ref{dieprop}.
We restrict to the interesting case $b\geq 3$, for $b=2$ similar
bounds can be given, see Remark~\ref{langeremark}. We again formulate Proposition~\ref{balddar} in a more general
way than needed by using the parameter $A$ which will be $A=1$ in the applications.

\begin{proposition}  \label{balddar}
Let $b\geq 3$ and $R\geq 1$ be integers and $A\in{\{1,2,\ldots,b-1\}}$.
Let $e,f$ be integers with $R<e<f$. Let $a\in{\mathbb{R}}$ with the property
\begin{equation}  \label{eq:grunzlager}
\frac{a}{b}\leq Ab^{R}(b^{-e}+b^{-f})<a.
\end{equation}
Define $\mathscr{H}$ the set of all numbers of the form
$\chi_{e+1}b^{-e-1}+\cdots+\chi_{f-1}b^{-f+1}$, where 
$\chi_{l}\in\{0,A\}$ for $e+1\leq l\leq f-1$.
Then, there exists $\kappa\in{\mathscr{H}}$ such that 
\begin{equation}  \label{eq:nochr}
\frac{1}{b-1}\cdot \frac{b^{f-e+1}-1}{b^{f-e+1}+b}\leq b^{R}\cdot \frac{\kappa+Ab^{-e}+Ab^{-f}}{a}<1.
\end{equation}
\end{proposition}

\begin{proof}
Define $\kappa$ as the largest element in $\mathscr{H}$ for which the right hand side
of \eqref{eq:nochr} holds. This is well-defined since $\kappa=0$ is a suitable choice
and in this case the right inequality in \eqref{eq:nochr} holds by assumption \eqref{eq:grunzlager}. 
Put $v:=Ab^{-e}+Ab^{-e-1}+\cdots+Ab^{-f}$ and $w:=Ab^{-e}+Ab^{-f}$. We separate two cases. 

Case 1: $a>b^{R}v$. Then $\kappa=v-w$ and the right inequality in \eqref{eq:nochr}
holds by construction. Moreover
\[
b^{R}\frac{\kappa+w}{a}=b^{R}\frac{v}{b}\frac{b}{a}\geq b^{R-1}v\frac{b^{-R}(b^{-e}+b^{-f})^{-1}}{A}
=\frac{vb^{-1}}{(b^{-e}+b^{-f})A}=\frac{\sum_{j=e}^{f} b^{-j}}{b(b^{-e}+b^{-f})},
\]
and the right hand side equals the left hand side in \eqref{eq:nochr}.

Case 2: $a\leq b^{R}v$. Then, by \eqref{eq:grunzlager}, there exists a largest index
$t\in{[e,f-2]}$ such that the inequality $b^{R}u_{t}<a$ holds, where $u_{t}:=Ab^{-e}+Ab^{-e-1}+\cdots+Ab^{-t}+Ab^{-f}$.
By definition $\kappa\geq u_{t}-w$ and the right inequality in \eqref{eq:nochr} holds.
Moreover, by maximality of $t$ we infer $a\leq b^{R}(Ab^{-e}+Ab^{-e-1}+\cdots+Ab^{-t}+Ab^{-t-1}+Ab^{-f})$ and further
\[
b^{R}\cdot \frac{\kappa+w}{a}\geq \frac{b^{R}u_{t}}{a}
\geq \frac{b^{-e}+b^{-e-1}+\cdots+b^{-t}+b^{-f}}{b^{-e}+b^{-e-1}+\cdots+b^{-t}+b^{-t-1}+b^{-f}}.
\]
The right hand side is easily seen to increase as $t$ does since
it can be written $1-b^{-t-1}/(b^{-e}+b^{-e-1}+\cdots+b^{-t-1}+b^{-f})$. Consequently for $b\geq 3$ indeed
\[
b^{R}\cdot \frac{\kappa+w}{a}\geq\frac{b^{-e}+b^{-f}}{b^{-e}+b^{-e-1}+b^{-f}}>\frac{b^{-e}}{b^{-e}+b^{-e-1}}
=\frac{b}{b+1}>\frac{1}{b-1}>\frac{1}{b-1}\cdot \frac{b^{f-e-1}-1}{b^{f-e-1}+b}.
\]
In the first strict inequality we used that for positive real numbers $y_{1},y_{2},y_{3}$ 
with $y_{1}<y_{2}$ we have $y_{1}/y_{2}<(y_{1}+y_{3})/(y_{2}+y_{3})$,
applied to $y_{1}=b^{-e}, y_{2}=b^{-e}+b^{-e-1}$ and $y_{3}=b^{-f}$.
\end{proof}

\section{Proof of Theorems~\ref{theor}, \ref{satz}} \label{beweis}

The considered sets $\mathscr{K}_{k}(.)$ are symmetric with respect to $0$,
for $(q,p_{1},p_{2},\ldots,p_{k})$ satisfies the inequality within the 
definition of $\mathscr{K}_{k}$ or $\mathscr{K}_{k}^{\ast}$ for $\zeta$ if and only if 
$(q,-p_{1},p_{2},\ldots,(-1)^{k}p_{k})$ satisfies the inequality for $-\zeta$.
Thus we may restrict to $q>0$ and $I\subseteq(0,1/2), J\subseteq{(0,\infty)}$ in the proofs.
This enables us to apply Proposition~\ref{eigenschaften}, which makes thing
a little less technical.
We use the abbreviation of an {\em everywhere uncountable set} for a set
that has uncountable intersection with any non-empty open interval of 
$\mathbb{R}$ in the proof. We point out that
the uncountable cardinality of real numbers that we will construct
within the proof of Theorem~\ref{theor} ensures that the restriction of $\mathscr{K}_{k}$ 
to numbers in $\mathscr{C}_{k}$, i.e. not algebraic of small degree, will not be relevant.

The proof of the first assertion of Theorem~\ref{theor}, where the result is stronger anyway, 
will be not too complicated to derive from Theorem~\ref{jth}
with aid of the results from Section~\ref{sektion1}.
The proof of the second assertion will be more technical, since we have to apply \eqref{eq:viel}
instead of \eqref{eq:sieben} when applying Lemma~\ref{lemma2}. As $L_{k}$ depends on $\zeta$, we 
get a weaker result in this case. The remainder term in \eqref{eq:viel} makes the proof of 
this assertion slightly more technical as well.

\begin{proof}[Proof of Theorem~\ref{theor}]
We start with the first claim.
Let $\Psi$ be arbitrary with \eqref{eq:bed4}, \eqref{eq:bed3}. 
Write $\Psi(x)=\Delta(x)x^{-1}$ with 
a function $\Delta(x)$ which obviously has the property
$\Delta(x)< 1/2$ for all large $x$ and moreover
\begin{equation}  \label{eq:ab}
\Delta(lx)(lx)^{-1} \leq l\cdot \Delta(x)x^{-1}, \qquad l,x\in{\mathbb{N}}.
\end{equation}
We want that \eqref{eq:jarn} is uncountable
for the function $\widetilde{\Psi}(x)=\Delta(x^{k})x^{-2k+1}$. We have to check that
\eqref{eq:bed3}, \eqref{eq:bed1} are satisfied for $\widetilde{\Psi}$ in order to apply 
Theorem~\ref{jth}.
Applying \eqref{eq:ab} to $x^{k},l^{k}$ leads after simplification to
\[
\Delta(l^{k}x^{k})\leq l^{2k}\Delta(x^{k}), \qquad l,x\in{\mathbb{N}}.
\]
This indeed yields
\[
\widetilde{\Psi}(lx)=\Delta(l^{k}x^{k})(lx)^{-2k+1}\leq
l\cdot \Delta(x^{k})x^{-2k+1}
=l\cdot \widetilde{\Psi}(x).
\]
On the other hand, since $k\geq 2$ we have $x^{-2k+1}=o(x^{-1})$, and since $\Delta(x)< 1/2$
this implies $\widetilde{\Psi}(x)=o(x^{-1})$ as $x\to\infty$, which we identify as \eqref{eq:bed1}.

Thus \eqref{eq:jarn} is indeed uncountable. In other words, 
there exists an uncountable set of $\zeta\in{I}$,
such that for any fixed $c<1$, we have
\begin{equation} \label{eq:plan}
c\widetilde{\Psi}(q)=c\Delta(q^{k})q^{-2k+1}<\Vert \zeta q\Vert <\Delta(q^{k})q^{-2k+1}=\widetilde{\Psi}(q)
\end{equation}
for arbitrarily large integers $q$, and 
\begin{equation} \label{eq:ismus}
\Vert \zeta q\Vert > c\Delta(q^{k})q^{-2k+1}=c\widetilde{\Psi}(q)
\end{equation}
for all sufficiently large integers $q\geq \hat{q}(\widetilde{\Psi},c)$. 
Using that $\zeta\in{(0,1/2)}$, the same choice of $\zeta$ 
will be suitable for the function $\Psi(x)$, as we shall show. This also implies
the effectiveness, since the proof of Theorem~\ref{jth} is constructive. 

Restricting to $\zeta\in{I\subseteq (0,1/2)}$, it follows from $\Delta(x^{k})<1/2$
for large $x$ and $2k-1>k$ that we may apply
\eqref{eq:sieben} from Lemma~\ref{lemma2} to the $q$ that satisfy \eqref{eq:plan}. It yields
for those $q$ the relation
$\max_{1\leq j\leq k}\Vert q^{k}\zeta^{j}\Vert=q^{k-1} \Vert \zeta q\Vert$.
Thus \eqref{eq:plan} further implies
\[
c\Psi(q^{k})=c\Delta(q^{k})q^{-k}<\max_{1\leq j\leq k} \Vert q^{k}\zeta^{j}\Vert<\Delta(q^{k})q^{-k}=\Psi(q^{k}).
\]
Hence, if we let $z:=q^{k}$, for any fixed $c<1$ and $z$ large enough indeed we have 
\begin{equation}  \label{eq:jarni}
c\Psi(z)=c\Delta(z)z^{-1}<\max_{1\leq j\leq k} \Vert z\zeta^{j}\Vert<\Delta(z)z^{-1}=\Psi(z).
\end{equation}
For the first assertion concerning $\mathscr{K}_{k}$, it remains to show
\begin{equation} \label{eq:dada}
\max_{1\leq j\leq k} \Vert z\zeta^{j}\Vert> c\Delta(z)z^{-1}=c\Psi(z)
\end{equation}
for all large integers $z$ (that are not necessarily $k$-th powers of an integer).
Assume the opposite, i.e. there exist arbitrarily large integers $z$ that violate \eqref{eq:dada}.
Recall $L_{k}(\zeta)=1$ for $\zeta\in{I}$ by Proposition~\ref{eigenschaften}.
Hence $c\Psi(x)<\Psi(x)<(1/2)x^{-1}$ for 
large $x\geq \hat{x}(\epsilon)$, and application of Corollary~\ref{lemma} to 
the function $c\Psi(x)$ implies that \eqref{eq:dada} is violated also for
arbitrarily large $z$ of the form $z=q^{k}$ and additionally $\Vert q \zeta\Vert\leq q^{-2k+1}$.
For the assertion on $\mathscr{K}_{k}^{\ast}$, in view of \eqref{eq:aufgfaedelt}
it is sufficient to consider such $z$ too.
To sum up, we obtain a sequence of values $q$ with the properties
\begin{eqnarray*}
\max_{1\leq j\leq k} \Vert q^{k}\zeta^{j}\Vert&\leq & c\Delta(q^{k})q^{-k}=c\Psi(q^{k})  \\
\Vert q\zeta\Vert &\leq& q^{-2k+1}. 
\end{eqnarray*}
Since $q^{-2k+1}\leq (1/2)q^{-k+1}$ for $q>1$, we may apply Lemma~\ref{lemma2},
more precisely \eqref{eq:sieben} since $\zeta\in{(0,1/2)}$ . It yields
\[
\Vert q \zeta\Vert= q^{1-k}\max_{1\leq j\leq k} \Vert q^{k}\zeta^{j}\Vert 
\leq c\Delta(q^{k})q^{-2k+1} =c\widetilde{\Psi}(q)
\]
contradicting \eqref{eq:ismus}. Hence \eqref{eq:dada} holds and the first assertion is proved.

We show the second claim.
Consider $c_{0}<1$, an open interval $J$ which we can assume to be bounded,
and a function $\Psi$ satisfying \eqref{eq:bed3} and
\eqref{eq:bed5} for some $d$ to be determined later, fixed.
Write $\Psi(x)=\Delta(x)x^{-1}$ with a function $\Delta(x)\leq dx^{-1}$ for large $x$. 
Let $c\in{(c_{0},1)}$ be arbitrary. 
Pick any $\zeta_{0}\in{J}$ and define $L^{0}:=L_{k}(\zeta_{0})$ with $L_{k}$ from
Lemma~\ref{lemma2}. We may assume $L^{0}>1$, otherwise by Proposition~\ref{eigenschaften} 
we have $\zeta_{0}\in{J\cap (0,1/2)}$,
in particular $J$ contains a subinterval of $(0,1/2)$ and the claim follows
from the first part of the theorem. 

Define $\hat{\Psi}(x):=L^{0}\widetilde{\Psi}(x)$ with 
$\widetilde{\Psi}(x)=\Delta(x^{k})x^{-2k+1}$ as above. Obviously $\hat{\Psi}$
satisfies the conditions of Theorem~\ref{jth}, for the same reasons as $\widetilde{\Psi}$. 
Similar to \eqref{eq:plan}, \eqref{eq:ismus}, we obtain
\begin{equation} \label{eq:plann}
c\hat{\Psi}(q)=cL^{0}\Delta(q^{k})q^{-2k+1}<
\Vert \zeta q\Vert <L^{0}\Delta(q^{k})q^{-2k+1}=\hat{\Psi}(q)
\end{equation}
for arbitrarily large integers $q$, and 
\begin{equation} \label{eq:ismuss}
\Vert \zeta q\Vert > cL^{0}\Delta(q^{k})q^{-2k+1}=c\hat{\Psi}(q)
\end{equation}
for all $q\geq \hat{q}(\hat{\Psi},c)$, for an uncountable set of values $\zeta\in{J}$.    

Note that non-empty pre-images of real open intervals under monotonic continuous maps are open intervals again 
and thus have uncountable intersection with any everywhere uncountable set. 
Moreover, if the function is strictly increasing, then the pre-image is locally a well-defined 
strictly increasing continuous function.  
In view of Proposition~\ref{eigenschaften},
we may apply this to the map $\zeta\mapsto L_{k}(\zeta)$, the
uncountable set of values $\zeta\in{J}$ constructed above and intervals $(L^{0},L^{0}+\delta)$
for small $\delta>0$. By $L^{0}>1$,
it yields that for arbitrarily small fixed $\delta>0$, for arbitrarily small $\nu>0$ 
we can choose uncountably many $\zeta_{1}\in{(\zeta_{0},\zeta_{0}+\nu)}$ that
satisfy the above conditions \eqref{eq:plann}, \eqref{eq:ismuss}, 
and have the additional property $L^{0}<L^{1}<L^{0}+\delta$
with $L^{1}:=L_{k}(\zeta_{1})$. Making $\nu$ smaller if necessary, we may assume $\zeta_{1}\in{J}$.
We will treat $\zeta_{1}$ as fixed in the sequel, so $L^{1}$ is fixed too.

Clearly we may apply \eqref{eq:viel} to the integers $q$ that satisfy \eqref{eq:plann}.
Denoting the involved remainder terms by $\epsilon(q)$, we infer
\[
L^{0}c\Delta(q^{k})q^{-k}=(L^{1}+\epsilon(q))
\max_{1\leq j\leq k} \Vert q^{k}\zeta_{1}^{j}\Vert<L^{0}\Delta(q^{k})q^{-k}
\]
which we can rewrite as
\begin{equation} \label{eq:11}
\frac{L^{0}}{L^{1}+\epsilon(q)}c\Psi(q^{k})<\max_{1\leq j\leq k} \Vert q^{k}\zeta_{1}^{j}\Vert
<\frac{L^{0}}{L^{1}+\epsilon(q)}\Psi(q^{k}).
\end{equation}
Note that since $L^{0}<L^{1}$ and $\epsilon(q)=o(1)$ as $q\to\infty$, 
we have $L^{0}/(L^{1}+\epsilon(q))<1$ for large $q$ and the quotients tend to $L^{0}/L^{1}$ as $q\to\infty$.
We can still choose the parameter $\delta$, and the quotient $L^{0}/L^{1}$ tends to $1$ as $\delta\to 0$.
Since we have strict inequality $c>c_{0}$ and $\epsilon(q)=o(1)$, choosing
$\delta$ in dependence of $c,c_{0}$ sufficiently small, 
putting $z:=q^{k}$ from \eqref{eq:11} we indeed infer
\[
c_{0}\Psi(z)<\max_{1\leq j\leq k} \Vert z\zeta_{1}^{j}\Vert<\Psi(z)
\]
for arbitrarily large integers $z$. It remains to prove
\begin{equation} \label{eq:abschaetzung}
\max_{1\leq j\leq k} \Vert z\zeta_{1}^{j}\Vert> c_{0}\Psi(z)
\end{equation}
for all sufficiently large integers $z$ (not necessarily $k$-th powers). As in the first assertion,
assume the opposite. Let $d$ in \eqref{eq:bed5} be arbitrary in the interval
\[
0<d<\frac{1}{2\cdot \sup_{t\in{J}}L_{k}(t)},
\]
which is equivalent to $0<d<(1/2)L_{k}(\gamma)^{-1}$ for $\gamma=\sup J$
by Proposition~\ref{eigenschaften}. Again Corollary~\ref{lemma} applied to the function 
$c_{0}\Psi(x)$, which is smaller than $(1/2)L_{k}(\zeta)^{-1}x^{-1}$ for any $\zeta\in{J}$
and large $x$, yields that \eqref{eq:abschaetzung} is violated for arbitrarily large
$z=q^{k}$ and additionally $\Vert q\zeta_{1}\Vert \leq q^{-2k+1}$. We further obtain
\begin{eqnarray*}
\max_{1\leq j\leq k} \Vert q^{k}\zeta_{1}^{j}\Vert&\leq & c_{0}\Delta(q^{k})q^{-k}=c_{0}\Psi(q^{k})  \\
\Vert q\zeta_{1}\Vert &\leq& q^{-2k+1}
\end{eqnarray*}
for arbitrarily large integers $q$. Again we may apply \eqref{eq:viel} to get 
\begin{equation}\label{eq:traum}
\Vert q \zeta_{1}\Vert= \frac{q^{1-k}}{L^{1}+\epsilon(q)}\max_{1\leq j\leq k} \Vert q^{k}\zeta_{1}^{j}\Vert 
\leq \frac{c_{0}}{L^{1}+\epsilon(q)}\Delta(q^{k})q^{-2k+1} =\frac{c_{0}}{L^{1}+\epsilon(q)}\widetilde{\Psi}(q).
\end{equation}
Recall \eqref{eq:ismuss} holds for $\zeta=\zeta_{1}$, so combination with \eqref{eq:traum} yields
\[
cL^{0}<\frac{c_{0}}{L^{1}+\epsilon(q)}.
\]
Since we have $c>c_{0}, L^{1}>L^{0}>1$ and $\epsilon(q)$ tends to $0$,
this cannot hold for large $q$. Again we conclude \eqref{eq:abschaetzung}.
As $c_{0},J,\Psi$ were arbitrary under the given restrictions
and we have shown the above can be done for uncountably many $\zeta_{1}\in{J}$,
the second assertion is proved.
\end{proof}

As indicated in Section~1, Theorem~\ref{satz} is established very similar. 
Recall that in Theorem~\ref{satz} the function $\Psi$ is defined on $\mathbb{R}_{>0}$,
which guarantees that all quantities that will appear are well-defined.

\begin{proof} [Proof of Theorem~\ref{satz}]
For given $\Psi$ as in the theorem, for any positive integer $j$ define 
\[
\Delta_{j}(x)=\Psi(x^{1/j})x^{(2j-1)/j},
\]
such that $\Psi(x)=\Delta_{j}(x^{j})x^{-2j+1}$ and $\Psi_{j}(x)=\Delta_{j}(x)x^{-1}$.
First we show the properties of $\Psi_{j}$. By \eqref{eq:admissible2} we infer $\Delta_{j}(x)=o(1)$ 
and hence indeed $\Psi_{j}(x)=o(x^{-1})$ as $x\to\infty$ for $1\leq j\leq k$.
Similarly, it is easy to check that strictly admissible $\Psi$ gives rise to
$\Psi_{j}$ satisfying \eqref{eq:admissible2} for any $k$. We show $\Psi_{j}$
is non-increasing. Since any map $t\mapsto t^{1/j}$ increases on the positive reals,
for arbitrary $0<x\leq y$ the estimate \eqref{eq:admissible1} implies
\[
\frac{\Psi(x^{1/j})}{\Psi(y^{1/j})}\geq \left(\frac{y}{x}\right)^{(j-1)/j}, \qquad 1\leq j\leq k.
\] 
This indeed leads to
\[
\Psi_{j}(x)=\Psi_{j}(x^{1/j})x^{(j-1)/j}\geq \Psi_{j}(y^{1/j})y^{(j-1)/j}=\Psi_{j}(y).
\]
It remains to prove that $\Psi_{j}$ satisfy a \eqref{eq:admissible1} type relation 
for any exponent $\mu>0$, provided that $\Psi$ has this property. Let $\mu>0$ arbitrary 
and put $\eta=k\mu+k-1$. By strict admissibility of $\Psi$, we have
\[
\frac{\Psi(x^{1/k})}{\Psi(y^{1/k})}\geq \left(\frac{y}{x}\right)^{\eta/k}
\]
for all large $x_{0}(\eta)<x\leq y$, and further
\[
\frac{\Psi_{j}(x)}{\Psi_{j}(y)}= \frac{\Psi(x^{1/k})x^{(k-1)/k}}{\Psi(y^{1/k})y^{(k-1)/k}}
\geq \left(\frac{y}{x}\right)^{(\eta+1-k)/k}=\left(\frac{y}{x}\right)^{\mu}.
\]

Recall that being an admissible function, $\Psi$ satisfies \eqref{eq:bed1}, \eqref{eq:bed2}
and since \eqref{eq:bed2} implies \eqref{eq:bed3}, Theorem~\ref{jth}
holds for $\Psi$. Moreover, we just proved that 
the functions $\Psi_{j}$ satisfy the properties of the function $\Psi$ in Theorem~\ref{theor}
since \eqref{eq:bed4}, \eqref{eq:bed5} both imply \eqref{eq:bed1}, and \eqref{eq:bed2} implies
\eqref{eq:bed3}. Hence, we can now proceed as in the proof of 
the first assertion of Theorem~\ref{theor}, where the $\Psi_{j}$ play the role of $\Psi$ from
Theorem~\ref{theor} and the present $\Psi$ the role of $\widetilde{\Psi}$,
to infer that any set $(-1/2,1/2)\cap \mathscr{K}^{j}$ with $\mathscr{K}^{j}$ the difference set
in \eqref{eq:lio}, contains the set $(-1/2,1/2)\cap \mathscr{K}^{1}$. Reversing the proof
of Theorem~\ref{theor} with Lemma~\ref{lemmaalt} shows that there is actually equality,
we omit the details as they are not of interest concerning Corollary~\ref{corkor} anyway. 
\end{proof}

\section{Proof of Theorems~\ref{begeisterung},~\ref{kgleicheins} } \label{vier} 

We first prove Theorem~\ref{begeisterung}.
For the convenience of the reader, we first give a detailed proof of the weaker assertion 
with bound $c<1/b$, where the outline of the proof is easier to detect. Subsequently, we will
describe how to generalize the result with Proposition~\ref{balddar}, where we will only sketch some parts that
can be deduced very similarly as for the weaker assertion. This avoids an overly technical proof.

\begin{proof} [Proof of Theorem~\ref{begeisterung}]
Recall we assume $J(b)=\{0,1\}$, such that $A=1$ in the sense of Lemma~\ref{dieprop}
and Proposition~\ref{balddar}. For technical reasons 
we first prove weaker the result Theorem~\ref{jox} for $c<1/b$ in the union. This
only requires Lemma~\ref{minkow} but neither the more general Lemma~\ref{dieprop}
nor Proposition~\ref{balddar}.

Consider an at the moment arbitrary function $\delta: \mathbb{N}\mapsto \mathbb{R}_{>0}$ 
which tends to $0$, to be determined later. Write also $\delta_{n}:=\delta(n)$.
Define the disturbed function 
\begin{equation}  \label{eq:schlanget}
\widetilde{\Psi}(x)=(1+\delta(x))\Psi(x).
\end{equation}
Put 
\begin{equation}  \label{eq:zetat}
\zeta= \sum_{n\geq 1} b^{-a_{n}}
\end{equation}
with an increasing sequence of positive integers $a_{n}$ depending on $\Psi$,
defined recursively. First let $a_{1}\geq 3$ be arbitrary. Since $b\geq 2$, 
construction \eqref{eq:zetat} implies $\zeta\in{(0,1/2)}$.
Now determine $a_{n+1}$ by
\begin{equation} \label{eq:zt}
\frac{\widetilde{\Psi}(b^{ka_{n}})}{b}\leq b^{ka_{n}-a_{n+1}}<\widetilde{\Psi}(b^{ka_{n}}).
\end{equation}
Note that $a_{n+1}$ is almost independent from the exact choice of the function $\delta$ for large $n$,
since a closer look at \eqref{eq:zt} shows a small perturbation of 
$\widetilde{\Psi}$ can effect a change of $a_{n+1}$ by at most $1$. 
By \eqref{eq:bed4} we have 
\begin{equation}  \label{eq:nbedt}
a_{n+1}\geq 2ka_{n}, \qquad n\geq n_{0}.
\end{equation}
Next we prove
\begin{equation}  \label{eq:yt}
\max_{1\leq j\leq k} \Vert \zeta^{j}b^{ka_{n}}\Vert=
b^{ka_{n}-a_{n+1}}(1+O(b^{-a_{n+1}}))
\end{equation}   
as $n\to\infty$. Write $\zeta=S_{n}+\epsilon_{n}$ with 
\[
S_{n}=\sum_{i=1}^{n} b^{-a_{i}}, \qquad \epsilon_{n}=\sum_{i=n+1}^{\infty} b^{-a_{i}}.
\]
Since $S_{n}<1, \epsilon_{n}<1$ and the binomial coefficients are bounded above by $k!$ , we have that 
\begin{equation} \label{eq:lieder}
\zeta^{j}=\sum_{i=0}^{j} \binom{j}{i}S_{n}^{i}\epsilon_{n}^{j-i}=
S_{n}^{j}+jS_{n}^{j-1}\epsilon_{n}+O(\epsilon_{n}^{2}), \qquad 1\leq j\leq k, 
\end{equation}
as $n\to \infty$. Note now that $b^{ka_{n}}S_{n}^{j}$ is an integer for $1\leq j\leq k$ by 
construction. Moreover, the remaining terms converge to $0$ and since
$0<S_{n}<\zeta<1/2$ are maximized for $j=1$ by Proposition~\ref{eigenschaften} at least for large $n$,  
and for $j=1$ clearly $jS_{n}^{j-1}=1$.
Thus in view of \eqref{eq:nbedt} and \eqref{eq:lieder} we have
\[
\max_{1\leq j\leq k}\left\Vert b^{ka_{n}}\zeta^{j}\right\Vert=
\left\Vert b^{ka_{n}}\zeta\right\Vert= b^{ka_{n}-a_{n+1}}(1+O(b^{-a_{n+1}})),
\]
so \eqref{eq:yt} is proved.

Denote the remainder terms in \eqref{eq:yt} as a sequence $(\rho_{n})_{n\geq 1}$ 
of positive reals. This sequence tends to $0$, in fact $\rho_{n}\asymp \epsilon_{n}$,
and by the above remarks almost independent of 
the exact choice of the function $\delta$. Let $z_{n}:=b^{ka_{n}}$. 
Combination of \eqref{eq:zt}, \eqref{eq:yt} shows
\begin{equation} \label{eq:vorhert}
\frac{1+\rho_{n}}{b}\widetilde{\Psi}(z_{n})\leq 
\max_{1\leq j\leq k} \Vert \zeta^{j}z_{n}\Vert< \widetilde{\Psi}(z_{n})(1+\rho_{n}).
\end{equation}
Let $\delta$ tend to $0$ sufficiently slowly such that
$\rho_{n}<\delta_{n}$. For any $\sigma>0$ and large enough $n\geq \hat{n}(\sigma)$
we have $b^{-1}(1+\delta_{n})^{-1}(1+\rho_{n})>(b+\sigma)^{-1}$. 
Inserting this in \eqref{eq:vorhert} in view of $\delta(x)\to 0$ in \eqref{eq:schlanget} yields 
\begin{equation} \label{eq:wahnt}
\frac{1}{b+\sigma} \Psi(z_{n})\leq \max_{1\leq j\leq k} \Vert \zeta^{j}z_{n}\Vert<\Psi(z_{n}), 
\qquad n\geq \hat{n}(\sigma).
\end{equation}
It remains to be shown that 
\begin{equation} \label{eq:sinnt}
\frac{1}{b+\sigma} \Psi(z)\leq \max_{1\leq j\leq k} \Vert \zeta^{j}z\Vert
\end{equation}
for all sufficiently large integers $z\geq \hat{z}(\sigma)$. 
By the assumption \eqref{eq:bed3}, due to Corollary~\ref{lemma}
we may restrict to $z=q^{k}$ where $q$ is a denominator
of the continued fraction expansion of $\zeta$ and $\Vert q\zeta\Vert\leq q^{-2k+1}\leq q^{-3}$
(here we need $k\geq 2$).
It is not hard to check only values $q$ of the form $q=b^{a_{n}}$ have this property.
Concretely it follows from \eqref{eq:nbedt} and Lemma~\ref{dieprop} 
(or Lemma~\ref{minkow}) since $2k/(2k-1)\leq 4/3<3$, for $k\geq 2$. 
The implied numbers $z=q^{k}$ are just $z_{n}$ as above, 
for which we have shown \eqref{eq:wahnt}, though.   

In view of \eqref{eq:nbedt} all constructed numbers $\zeta$ are very well approximable and thus 
transcendental (note that $\mathscr{K}_{k}$ is defined for $\mathscr{C}_{k}$ only). 
Finally, we modify the construction to show the set of such $\zeta$ is indeed uncountable. 
Define $a_{2n+1}$ given $a_{2n}$ as in \eqref{eq:zt}, but 
$a_{2n}\in{[(k+1)a_{2n-1},2ka_{2n-1}-1]}$ arbitrary, and define $\zeta$ by \eqref{eq:zetat}. 
Indeed, any $z_{2n}$ satisfies \eqref{eq:wahnt}, and
by virtue of Corollary~\ref{lemma}, Lemma~\ref{minkow} and \eqref{eq:bed4},
we infer \eqref{eq:sinnt} very similarly, since still $(k+1)/k\leq 3/2<3$.
Clearly this method yields uncountably many such numbers.
This finishes the proof of the weaker claim where the union is taken over $c<1/b$, i.e. Theorem~\ref{jox}.

For the stronger result, consider $\zeta\in{(0,1/2)}$ with base $b$ digits in $\{0,1\}$ 
whose $1$ digits are not isolated as in the proof above, but are at decimal places $a_{n}$ with a 
sequence $(a_{n})_{n\geq 1}$ basically as in Lemma~\ref{dieprop}.
For any $n\in{\mathbb{N}}$ we will define an integer block 
$I_{n}=\{e_{n},e_{n}+1,\ldots,f_{n}\}$, and require the $e_{n}$-th and $f_{n}$-th base $b$ digit of $\zeta$
to equal $1$, whereas for now there is free base $b$ digit choice $0$ or $1$ 
within $I_{n}\setminus{\{e_{n},f_{n}\}}$. Put $I=\cup I_{n}$ and put $0$ in the
base $b$ decimal places within $\mathbb{N}\setminus I$, which means at places 
of the form $f_{n}+1, f_{n}+2,\ldots,e_{n+1}-1$.  
Suppose that the lengths of $I_{n}$ are given as $f_{n}-e_{n}=n$, such that they
tend to infinity but rather slowly. 
Let $e_{n+1}$ be defined recursively from $f_{n}$ via
\begin{equation}  \label{eq:zzt}
\frac{\widetilde{\Psi}(b^{kf_{n}})}{b}\leq b^{kf_{n}-e_{n+1}}+b^{kf_{n}-f_{n+1}}<\widetilde{\Psi}(b^{kf_{n}}).
\end{equation}
Given any $I_{1}=\{e_{1},e_{1}+1\}$ with large $e_{1}$, the sets $I_{n}$ are now well-defined
and disjoint. Consider the class $\mathscr{A}$ of arising numbers $\zeta= \sum_{n\geq 1} b^{-a_{n}}$ constructed 
as above, that is $(a_{n})_{n\geq 1}$ is strictly increasing, contains $(e_{n})_{n\geq 1}$ and $(f_{n})_{n\geq 1}$
as subsequences and all $a_{n}$ belong to $I$. 
Any $\zeta$ in $\mathscr{A}$ has a representation
\begin{equation} \label{eq:zeter}
\zeta=\sum_{n=1}^{\infty} b^{-e_{n}}+\sum_{n=1}^{\infty} b^{-f_{n}}+\sum_{n=1}^{\infty} \kappa_{n}
\end{equation}
for $\kappa_{n}$ rational numbers of the form 
\begin{equation} \label{eq:kappa}
\kappa_{n}:=\chi_{e_{n+1}+1}b^{-e_{n+1}-1}+\chi_{e_{n+1}+2}b^{-e_{n+1}-2}+\cdots+\chi_{f_{n+1}-1}b^{-f_{n+1}+1},
\end{equation}
where $\chi_{l}\in\{0,1\}$ for $e_{n+1}+1\leq l\leq f_{n+1}-1$. 
We will recursively determine the numbers $\kappa_{n}$ for which the arising
$\zeta\in{\mathscr{A}}$ defined by \eqref{eq:zeter} will suit our claim.
First observe by \eqref{eq:zzt} and \eqref{eq:bed4} again we have 
\begin{equation}  \label{eq:nbedit}
e_{n+1}\geq 2kf_{n}, \qquad n\geq n_{0}.
\end{equation}
Thus all assumptions of Lemma~\ref{dieprop} with $A=1$, any $\omega<2k$ and
the sequences $(e_{n})_{n\geq 1},(f_{n})_{n\geq 1}$
and any arising sequence $(a_{n})_{n\geq 1}$ are satisfied. 
For the moment,
let $\kappa_{n}$ be any rational number in $[0,1)$
of the form \eqref{eq:kappa}
where $\chi_{l}\in\{0,1\}$ for $e_{n+1}+1\leq l\leq f_{n+1}-1$, which we will specify soon. 
Recall we put $1$ for the $e_{n+1}$-th and $f_{n+1}$-th base $b$ digit of $\zeta$.
Hence any such choice of $\kappa_{n}$ determines the choice of $0$ and $1$-values in $I_{n+1}$.
For any $\zeta$ in $\mathscr{A}$, similarly to \eqref{eq:yt}
we infer
\begin{equation}  \label{eq:yyt}
\max_{1\leq j\leq k} \Vert \zeta^{j}b^{kf_{n}}\Vert=
b^{kf_{n}}(b^{-e_{n+1}}+b^{-f_{n+1}}+\kappa_{n}+O(b^{-2e_{n+1}}))
\end{equation}  
as $n\to\infty$ with positive remainder term which is of much smaller order than
$b^{-e_{n+1}}+\kappa_{n}$ by the assumptions on the sequences $(e_{n})_{n\geq 1}, (f_{n})_{n\geq 1}$. 
For every $n$ choose $\kappa_{n}$ of the form \eqref{eq:kappa} largest possible such that 
\begin{equation}  \label{eq:eurot}
\frac{\widetilde{\Psi}(b^{kf_{n}})}{b}\leq b^{kf_{n}}(b^{-e_{n+1}}+b^{-f_{n+1}}+\kappa_{n})<\widetilde{\Psi}(b^{kf_{n}}).
\end{equation}
Such a choice is clearly possible, for if we let all digits within $I_{n+1}\setminus \{e_{n+1},f_{n+1}\}$ vanish
and thus $\kappa_{n}=0$, the estimate \eqref{eq:eurot} follows from \eqref{eq:zzt}.
We need a better lower bound for the 
quotient $b^{kf_{n}}(b^{-e_{n+1}}+b^{-f_{n+1}}+\kappa_{n})/\widetilde{\Psi}(b^{kf_{n}})$.
We apply Proposition~\ref{balddar} with 
\[
A=1, \quad a:=\widetilde{\Psi}(b^{kf_{n}}), \quad R:=kf_{n}, \quad
e:=e_{n+1}, \quad f:=f_{n+1}.
\]
It yields
\begin{equation} \label{eq:oktobert}
\frac{b^{kf_{n}}(b^{-e_{n+1}}+b^{-f_{n+1}}+\kappa_{n})}{\widetilde{\Psi}(b^{kf_{n}})}\geq 
\frac{1}{b-1}\cdot \frac{b^{f_{n+1}-e_{n+1}+1}-1}{b^{f_{n+1}-e_{n+1}+1}+b},
\end{equation}
where the worst case scenario is that $b^{kf_{n}-e_{n+1}}$
is very close to the lower bound $\widetilde{\Psi}(b^{kf_{n}})/b$ 
in \eqref{eq:zzt} or equivalently $\widetilde{\Psi}(b^{kf_{n}})$
is close to $b^{kf_{n}-e_{n+1}+1}$, and in this case the optimal choice is
$\kappa_{n}=b^{-e_{n+1}-1}+b^{-e_{n+1}-2}+\cdots+b^{-f_{n+1}}$.

Since $f_{n}-e_{n}$ tends to infinity, the right hand side expression in \eqref{eq:oktobert}
tends to $1/(b-1)$. Similarly to the special case $c<1/b$, it follows from \eqref{eq:yyt} and
\eqref{eq:schlanget} if $\delta(x)\to 0$ sufficiently slow, that for any $\sigma>0$ 
there are arbitrarily large integers $z_{n}=b^{kf_{n}}$ for which
\begin{equation} \label{eq:wahnsinnt}
\frac{1}{b-1+\sigma} \Psi(z_{n})\leq \max_{1\leq j\leq k} \Vert \zeta^{j}z_{n}\Vert<\Psi(z_{n}).
\end{equation}
Finally, the relation
\begin{equation} \label{eq:sinnwahnt}
\frac{1}{b-1+\sigma} \Psi(z)\leq \max_{1\leq j\leq k} \Vert \zeta^{j}z\Vert
\end{equation}
for all sufficiently large integers $z\geq \hat{z}(\sigma)$ must be inferred.
We proceed very similarly to the special case $c<1/b$ using Proposition~\ref{lemma} and Lemma~\ref{dieprop}.
Assume the contrary, that
\begin{equation} \label{eq:hypothesis}
\max_{1\leq j\leq k} \Vert \zeta^{j} z\Vert<\frac{1}{b-1+\sigma} \Psi(z)
\end{equation}
has arbitrarily large solutions $z$.
Then again by Corollary~\ref{lemma} for any fixed $\epsilon>0$ and any such (large) $z$ the estimate
$\Vert \zeta z\Vert=\vert \zeta z-h\vert<z^{-2k+1+\epsilon}$ is satisfied, where we write $h$ for the closest integer 
to $\zeta z$. Then $(h,z)$ must be an integral multiple of
some $(b^{f_{n}},\lfloor b^{f_{n}}\zeta\rfloor)$ by \eqref{eq:zeus} in context of Lemma~\ref{dieprop}. Observe
that these vectors are primitive since $\lfloor b^{f_{n}}\zeta\rfloor\equiv 1 \bmod b$, which follows
from the fact that the $f_{n}$-th base $b$ digit of $\zeta$ is $1$ by construction. 
Hence from the structural claim of Corollary~\ref{lemma} we further deduce that any solution of \eqref{eq:hypothesis} 
must be of the form 
\[
M(b^{kf_{n}},b^{(k-1)f_{n}}\lfloor b^{f_{n}}\zeta\rfloor,b^{(k-2)f_{n}}\lfloor b^{f_{n}}\zeta\rfloor^{2},
\ldots, \lfloor b^{f_{n}}\zeta\rfloor^{k}), \quad M\in{\mathbb{Z}}.
\]
For $M=1$ clearly the minimum of 
$\max_{1\leq j\leq k} \Vert Mb^{f_{n}}\zeta^{j}\Vert=M\max_{1\leq j\leq k} \Vert b^{f_{n}}\zeta^{j}\Vert$ is 
obtained (see also Corollary~\ref{lemma})
and also for $M=1$ the first coordinate coincides with $z_{n}$ above. 
Hence by \eqref{eq:bed3} if \eqref{eq:hypothesis} would have a solution
then there would also be one induced by $z=z_{n}$. However, for $z=z_{n}$ we proved \eqref{eq:wahnsinnt},
contradiction.  

The modification to obtain uncountably many suitable $\zeta$ is performed similarly
to the special case $c<1/b$ as well by considering a single additional element $a_{j}$ 
at suitable places between every second pair $I_{2n},I_{2n+1}$ of two consecutive blocks.
\end{proof}

\begin{remark} \label{langeremark}
Let $b=2$. A very similar proof works and yields another proof the first claim of 
Theorem~\ref{theor}, where the binary digit expansion of the 
implied $\zeta$ instead of the continued fraction expansion is determined. 
Concretely, if we let $b=2$ within the assumptions of Proposition~\ref{dieprop},
the proof of its case 1 works and is applicable to the proof of
Theorem~\ref{begeisterung} precisely as for $b\geq 3$. 
For the concern of case 2, we obtain a bound as follows.
By the assumption of case 2, the binary expansion of $a$ is given as 
$a=\tau_{0}2^{R-e}+\tau_{1}2^{R-e-1}+\cdots$ with 
$\tau_{l}\in{\{0,1\}}$ for all $l$, and $\tau_{0}=1$ by \eqref{eq:grunzlager}. 
Since $\chi_{l}$ are arbitrary in $\{0,1\}$, we may put
$\chi_{e+1}=\tau_{1},\chi_{e+2}=\tau_{2},\ldots,\chi_{f-1}=\tau_{f-e-1}$. Then
\begin{align*}
2^{R}\cdot \frac{\kappa+2^{-e}+2^{-f}}{a}&=
\frac{\tau_{0}2^{-e}+\tau_{1}2^{-e-1}+\cdots+\tau_{f-e}2^{-f}}{\tau_{0}2^{-e}+\tau_{1}2^{-e-1}+\cdots}  \\
&\geq \frac{2^{-e}+\tau_{1}2^{-e-1}+\cdots+\tau_{f-e-1}2^{-f+1}}
{2^{-e}+\tau_{1}2^{-e-1}+\cdots+\tau_{f-e-1}2^{-f+1}+2^{-f}}.
\end{align*}
The right hand side is smallest if all $\tau_{j}$ of positive index vanish, and thus
\begin{equation} \label{eq:aehnlich}
2^{R}\cdot \frac{\kappa+2^{-e}+2^{-f}}{a}\geq \frac{2^{-e}}{2^{-e}+2^{-f}}=\frac{2^{f-e}}{2^{f-e}+1}.
\end{equation}
The most right expression tends to $1$ as $f-e$ tends to infinity, such that we can apply
\eqref{eq:aehnlich} similarly to \eqref{eq:oktobert} in the proof of Theorem~\ref{begeisterung}. 
\end{remark}

\begin{remark}
Functions $\Psi$ that lead to what was called
the worst case scenario in the proof asymptotically for all large 
$q\in{\mathbb{N}}$ can readily be constructed, for example
\[
\Psi(x)=x^{-N}+\exp(-x), \qquad N\in{\mathbb{N}}.
\]
The bound $1/(b-1)$ seems to be close to the optimal value that
can be obtained with the current methods, in particular restricting
to approximation by rationals that belong to the missing digit set $K_{J(b)}$ as well.
\end{remark}

Now we establish Theorem~\ref{kgleicheins}.

\begin{proof}[Proof of Theorem~\ref{kgleicheins}]
By assumption \eqref{eq:staerker} we may assume $\Psi(x)<x^{-\theta}$ for some $\theta>\gamma$.
For the construction of suitable $\zeta$,
proceed as in the proof of Theorem~\ref{begeisterung} with $k=1$. Observe that the stronger
condition \eqref{eq:staerker} implies the stronger estimate 
\begin{equation} \label{eq:frech}
e_{n+1}\geq (\theta+1)f_{n}> 2f_{n}=2kf_{n}, \qquad n\geq n_{0},
\end{equation}
instead of \eqref{eq:nbedt}.
We can infer \eqref{eq:wahnsinnt} for $k=1$ and $z_{n}=b^{kf_{n}}=b^{f_{n}}$ precisely as in the case $k\geq 2$. 
Concerning \eqref{eq:sinnwahnt} for $k=1$, note that by \eqref{eq:frech} we may apply
Lemma~\ref{dieprop} with any $\omega<\theta+1$. On the other hand, an easy calculation
shows that our assumption $\theta>\gamma$ implies $\theta>(\theta+1)/\theta$. 
Hence, if $(x,y)$ is not a multiple of some $(x_{n},y_{n})$ as in Lemma~\ref{dieprop}, 
choosing $\epsilon>0$  sufficiently small and $\omega$ sufficiently close to $\theta+1$
and $x\geq x_{0}(\epsilon)$  sufficiently large, we have
\[
\vert \zeta x-y\vert\geq x^{-\frac{\omega}{\omega-1}-\epsilon}
\geq x^{-\frac{\theta+1}{\theta}-2\epsilon}> x^{-\theta}>\Psi(x)>\frac{1}{b-1+\sigma} \Psi(x).
\]
Finally, if $(x,y)$ is a multiple of some $(x_{n},y_{n})$,
the assumption \eqref{eq:bed3} and \eqref{eq:wahnsinnt} guarantee \eqref{eq:sinnwahnt} as well.
\end{proof}

We close with some remarks to Theorem~\ref{kgleicheins}.

\begin{remark}
Theorem~\ref{kgleicheins} extends to $b=2$ similarly to Remark~\ref{langeremark} and 
leads to explicit binary expansions of numbers $\zeta$ that satisfy Theorem~\ref{jth},
provided $\Psi$ satisfies \eqref{eq:staerker}.
\end{remark}

\begin{remark}
It is clear from the proof
that an improvement of the bound in Lemma~\ref{dieprop} readily allows an improvement 
of the bound in \eqref{eq:staerker} in Theorem~\ref{kgleicheins}. See also Remark~\ref{feeertig}.
With some concise combination of the block method of the proof of Theorem~\ref{kgleicheins} 
and the Folding Lemma instead of Lemma~\ref{dieprop}, 
it seems reasonable that \eqref{eq:bed1} is sufficient, which would 
unconditionally improve Theorem~\ref{bu}.
\end{remark}

The author warmly thanks the anonymous referee for pointing out some inaccuracies and providing various advices!

\end{document}